\documentclass{amsart}
\usepackage{latexsym,amsxtra,amscd,ifthen}
\usepackage{amsfonts}
\usepackage{verbatim}
\usepackage{amsmath}
\usepackage{amsthm}
\usepackage{amssymb}
\usepackage{mathtools}

\numberwithin{equation}{subsection}
\makeatletter
\def\tagform@#1{\maketag@@@{(\ignorespaces E{#1}\unskip\@@italiccorr)}}
\renewcommand{\eqref}[1]{\textup{{\normalfont(E\ref{#1}}\normalfont)}}
\makeatother
\newtagform{orig}{(}{)}
\theoremstyle{plain}
\newtheorem{theorem}{Theorem}[section]
\newtheorem{lemma}[theorem]{Lemma}
\newtheorem{proposition}[theorem]{Proposition}
\newtheorem{corollary}[theorem]{Corollary}

\theoremstyle{definition}
\newtheorem{definition}[theorem]{Definition}
\newtheorem{example}[theorem]{Example}

\newtheorem*{remark*}{Remark}

\newtheorem{questions}[theorem]{Questions}

\newtheorem{notation}[theorem]{Notation}

\DeclareMathOperator{\gldim}{gldim}

\DeclareMathOperator{\Ext}{Ext}

\DeclareMathOperator{\gr}{gr}

\DeclareMathOperator{\Aut}{Aut}

\DeclareMathOperator{\injdim}{injdim}

\DeclareMathOperator{\GKdim}{GKdim}

\newcommand{\fm}{\mathfrak{m}}

\newcommand{\mf}{\mathfrak}

\newcommand{\id}{\operatorname{id}}

\newcommand{\NN}{{\mathbb N}}
\newcommand{\QQ}{{\mathbb Q}}
\newcommand{\ZZ}{{\mathbb Z}}

\newcommand\cN{{\mathcal N}}

\newcommand\nat{{\rm($\natural$)}}

\newcommand\Btil{\widetilde{B}}
\newcommand\Gtil{\widetilde{G}}
\newcommand\Itil{\widetilde{I}}
\newcommand\Mtil{\widetilde{M}}
\newcommand\ptil{\widetilde{p}}

\newcommand\xtil{\widetilde{x}}
\newcommand\ytil{\widetilde{y}}
\newcommand\lambdatil{\widetilde{\lambda}}
\newcommand\pitil{\widetilde{\pi}}
\newcommand\chitil{\widetilde{\chi}}

\newcommand\co{\operatorname{co}}

\newcommand\eps{\varepsilon}
\newcommand\Znonneg{\ZZ_{\ge0}}
\newcommand\Zpos{\ZZ_{>0}}
\newcommand\kx{k^{\times}}
\newcommand\gfrak{{\mathfrak g}}
\newcommand\mfrak{{\mathfrak m}}

\DeclareMathOperator{\op}{op}
\DeclareMathOperator{\Id}{Id}
\DeclareMathOperator{\wgt}{wt}

\begin{document}

\title[Non-affine Hopf algebras]
{Non-affine Hopf algebra domains of\\
Gelfand-Kirillov dimension two}

\author{K.R. Goodearl and J.J. Zhang}

\address{Goodearl: Department of Mathematics 
University of California at Santa Barbara,
Santa Barbara, CA 93106}

\email{goodearl@math.ucsb.edu} 

\address{Zhang: Department of Mathematics, Box 354350,
University of Washington, Seattle, Washington 98195, USA}

\email{zhang@math.washington.edu}

\begin{abstract}
We classify all non-affine Hopf algebras $H$ over an 
algebraically closed field $k$ of characteristic zero that 
are integral domains of Gelfand-Kirillov dimension two and 
satisfy the condition $\Ext^1_H(k, k) \neq 0$. The affine ones 
were classified by the authors in 2010 \cite{GZ}.
\end{abstract}


%16T05 57T05 (2010-now) Hopf algebras and their applications
%16E65  (2000-now) Homological conditions on rings.
%16P90   (1991-now) Growth rate, Gelfand-Kirillov dimension


\maketitle

%%%%%%%%%%%%%%%%%%%%%%%%%%%%%%%%%%%%%%%%%%%%%%%%%%%%%%%%%%%%%%%%%%%%%

\setcounter{section}{-1}
%%%%%%%%%%%%%%%%%%%%%%%%%%%%%%%%%%%%%%%%%
%%%%%%%%%%%%%%%%%%%%%%%
\section{Introduction}
\label{xxsec0}

Throughout let $k$ be a base field that is algebraically closed  
of characteristic zero. All algebras and Hopf algebras are assumed to be $k$-algebras. The main result of 
\cite{GZ} is the classification of the affine Hopf $k$-algebras $H$ that 
are integral domains of Gelfand-Kirillov dimension two and satisfy the 
extra homological condition  
\begin{equation}
\label{E0.0.1}  \usetagform{orig}  \tag{$\natural$}
\Ext^1_H(k, k) \neq 0,
\end{equation}
where $k$ also denotes the trivial module $H/\ker \epsilon$. We say that $H$ 
is {\it affine} if it is finitely generated over $k$ as an algebra. 
Geometrically, the condition $(\natural)$ means that the tangent 
space of the corresponding quantum group is non-trivial. By 
\cite[Theorem 3.9]{GZ}, the condition $(\natural)$ is equivalent to 
the condition that the corresponding quantum group contains a classical 
algebraic subgroup of dimension one. The authors asked whether the 
condition $(\natural)$ is automatic when $H$ is an affine domain of 
Gelfand-Kirillov dimension (or GK-dimension, for short) two 
\cite[Question 0.3]{GZ}. This question was answered negatively in \cite{WZZ1}. 
Some affine Hopf algebra domains of GK-dimension two that do not satisfy 
$(\natural)$ were given and studied in \cite{WZZ1}.

All Hopf domains of GK-dimension one are listed in \cite[Proposition 2.1]{GZ}.
All affine Hopf domains of GK-dimension two satisfying $(\natural)$ are listed in
\cite[Theorem 0.1]{GZ}. The main goal of the present paper is to classify 
{\it non-affine} Hopf domains of GK-dimension two satisfying $(\natural)$. 
Together with \cite[Theorem 0.1]{GZ}, this provides a complete list of all
Hopf domains of GK-dimension two satisfying $(\natural)$:

\begin{theorem}
\label{xxthm0.1} 
Let $H$ be a Hopf domain of GK-dimension two satisfying 
$(\natural)$. Then it is isomorphic, as a Hopf algebra, to one of 
the following.
\begin{enumerate}
\item[(1)]
$kG$ where $G$ is a subgroup of ${\mathbb Q}^{2}$ 
containing ${\mathbb Z}^2$.
\item[(2)]
$kG$ where $G = L\rtimes R$ for some subgroup $L$ of 
${\mathbb Q}$ containing ${\mathbb Z}$
and some subgroup $R$ of $\ZZ_{(2)}$
containing ${\mathbb Z}$.
\item[(3)] $U(\mf{g})$ where $\mf{g}$ is a $2$-dimensional Lie algebra over $k$.
\item[(4)]
$A_{G}(e,\chi)$ where $G$ is a nonzero subgroup of
${\mathbb Q}$ {\rm{[Example \ref{xxex1.2}]}}.
\item[(5)]
$C_{G}(e,\tau)$ where $G$ is a nonzero subgroup of
${\mathbb Q}$ {\rm{[Example \ref{xxex1.3}]}}.
\item[(6)]
$B_G(\{p_i\},\chi)$ where $G$ is a nonzero subgroup of
${\mathbb Q}$ {\rm{[Construction \ref{xxsec2.1}]}}.
\end{enumerate}
\end{theorem}

In part (2) of the above theorem, $\ZZ_{(2)}$ 
denotes the localization of the ring $\ZZ$ at the maximal ideal 
$(2)$, that is, the ring of rational numbers with odd denominators. 

There are more Hopf domains of GKdim two if the hypothesis 
$(\natural)$ is removed from Theorem \ref{xxthm0.1}, see \cite{WZZ1}. 

We also study some algebraic properties of the algebras in
Theorem \ref{xxthm0.1}. The following is an easy consequence
of Theorem \ref{xxthm0.1}.

\begin{corollary}
\label{xxcor0.2} Let $H$ be as in Theorem {\rm\ref{xxthm0.1}}.
Then the following hold.
\begin{enumerate}
\item[(1)]
$H$ is pointed and generated by grouplike
and skew primitive elements. 
\item[(2)]
$H$ is countable dimensional over $k$. 
\item[(3)]
The antipode of $H$ is bijective.
\item[(4)]
Let $K$ be a Hopf subalgebra of $H$. Then $H_K$ and $_KH$ are free.
\item[(5)]
If $H$ is as in parts {\rm(1-5)} of Theorem {\rm\ref{xxthm0.1}},
then $2\leq \gldim H\leq 3$, 
while if $H$ is as in part {\rm(6)}, then $\gldim H = \infty$.
\end{enumerate}
\end{corollary}

By \cite[Proposition 0.2(b)]{GZ}, if $H$ in Theorem \ref{xxthm0.1} 
is noetherian, then $\injdim H=2$. So we conjecture that $\injdim H=3$
if $H$ in Theorem \ref{xxthm0.1} is non-noetherian.

There have been extensive research activities concerning infinite 
dimensional Hopf algebras (or quantum groups) in recent years. 
The current interests are mostly on noetherian and/or affine 
Hopf algebras. One appealing research direction is to 
understand some global structure of noetherian and/or affine 
and/or finite GK-dimensional Hopf algebras. 

A classical result of Gromov states that a finitely generated group $G$ 
has polynomial growth,  or equivalently, the associated group algebra 
has finite GK-dimension, if and only if $G$ has a nilpotent subgroup of 
finite index \cite{Gr}. So group algebras of finite GK-dimension are 
understood. It is natural to look for a Hopf analogue of this result, 
see \cite[Question 0.1]{WZZ2}. Another vague question is ``what can we 
say about a Hopf algebra of finite GK-dimension?''.  Let us mention a 
very nice result in this direction. Zhuang proved that every connected 
Hopf algebra of finite GK-dimension is a noetherian and affine domain 
with finite global dimension \cite{Zh}. Here the term ``connected'' 
means that the coradical is 1-dimensional. In general, the noetherian
and affine properties are not consequences of the finite GK-dimension
property. To have any sensible solution, we might restrict our attention 
to the domain case. A secondary goal of this paper is to promote 
research on Hopf domains of finite GK-dimension which are not
necessarily noetherian nor affine.

Let us start with some definitions.

\begin{definition}
\label{xxdef0.3}
Let $H$ be a Hopf algebra with antipode $S$.
\begin{enumerate}
\item[(1)]
$H$ is called {\it locally affine} if every finite subset of $H$ is 
contained in an affine Hopf subalgebra of $H$. 
\item[(2)]
$H$ is said to have {\it $S$-finite type} if there is a finite
dimensional subspace $V\subseteq H$ such that $H$ is generated
by $\bigcup_{i=0}^{\infty} S^i(V)$ as an algebra.
\item[(3)]
$H$ is said to satisfy (\emph{FF}) if for every Hopf subalgebra 
$K\subseteq H$, the $K$-modules $H_K$ and $_KH$ are faithfully flat.
\end{enumerate}
\end{definition}

It is clear that $H$ is affine if and only if $H$ is both
locally affine and of $S$-finite type. The question of whether
$H$ satisfies (FF) has several positive answers \cite{Ch, NZ, Ra1, 
Ra2, Ta1, Ta2}. In 1993, Montgomery asked if every Hopf 
algebra satisfies (FF) \cite[Question 3.5.4]{Mo}. A counterexample
was given in \cite{Sc}.  Hence Montgomery's question was modified 
to the Hopf algebras with bijective antipode. By a result of 
Skryabin \cite[Theorem A]{Sk}, every Hopf domain of finite
GK-dimension has bijective antipode. Prompted by 
Zhuang's result and Corollary \ref{xxcor0.2}, we ask

\begin{questions}
\label{xxque0.4}
Let $H$ be a Hopf domain of finite GK-dimension.
\begin{enumerate}
\item[(1)]
Is the $k$-dimension of $H$ countable, or equivalently, is $H$ 
countably generated as an algebra?
\item[(2)]
Is $H$ locally affine?
\item[(3)]
Is $H$ equal to the union of an ascending chain of affine Hopf subalgebras?
\item[(4)]
Is ``affine'' equivalent to ``noetherian''? See also \cite[Question 5.1]{WZ},
\cite[Questions D and E]{BG} and \cite[Question 2.4]{Go}.
\item[(5)]
Does $H$ satisfy (FF)?
\item[(6)]
Is $\injdim H$ bounded by a function of $\GKdim H$?
\item[(7)]
If $\gldim H$ is finite, is $\gldim H$ bounded by a function of $\GKdim H$?
\end{enumerate}
\end{questions}

If any of Questions  \ref{xxque0.4}(1,2,3) has a positive answer, that would
indicate that $H$ is somewhat close to being affine. We also have
the following result connecting some of these concepts. Note that 
pointed Hopf algebras satisfy (FF) by \cite{Ra2}.

\begin{theorem}
\label{xxthm0.5}
Let $H$ be a Hopf algebra that is left noetherian.
\begin{enumerate}
\item[(1)]
Suppose $H$ satisfies {\rm{(FF)}}. Then $H$ is of $S$-finite type. 
As a consequence, $\dim_k H$ is countable.
\item[(2)]
Suppose $H$ satisfies {\rm{(FF)}} and $H$ is locally affine.
Then $H$ is affine. 
\item[(3)]
If $H$ is pointed, then it is affine.
\end{enumerate}
\end{theorem}

Theorem \ref{xxthm0.5}(3) partially answers \cite[Question 5.1]{WZ}
in the pointed case, see also \cite[Question 2.4]{Go} and \cite[Question D]{BG}.

\subsection{Notation} 
\label{xxsec0.1}
Fix an algebraically closed base field $k$ of characteristic zero throughout.

Elements $u$ and $v$ of a $k$-algebra are said to \emph{quasi-commute} 
if $uv = qvu$ for some $q\in \kx$, in which case they are also said to 
\emph{$q$-commute}.

We shall reserve the term \emph{skew primitive} for $(1,g)$-skew 
primitive elements $z$, meaning that $g$ is grouplike and 
$\Delta(z) = z\otimes 1 + g\otimes z$. In this situation, $g$ 
is called the \emph{weight} of $z$, denoted $\wgt(z)$. 
General skew primitive elements can be normalized to the 
kind above, since if $w$ is $(a,b)$-skew primitive, then 
$a^{-1}w$ is $(1,a^{-1}b)$-skew primitive.

Let $G$ be an additive subgroup of $(\QQ,+)$. For  
convenience we sometimes identify it with the multiplicative 
$x$-power group, namely, 
$$G=\{x^{g}\mid g\in G\},$$
where $x^g x^h = x^{g+h}$ for $g,h\in G$ and $x^0=1$.
If $1\in G$, then we also write $x$ for $x^1$. Similarly, 
the group algebra $kG$ is identified with 
$\bigoplus_{g\in G} k \, x^g$. In Section \ref{xxsec2} we will
also use an additive submonoid $M\subseteq \QQ$. Then in this case, we identify
$M$ with the multiplicative $y$-power monoid, namely, 
$$M=\{y^{m}\mid m\in M\}.$$
By using these different notations, one sees the 
different roles played by $G$ and $M$ in Section \ref{xxsec2}. 

%%%%%%%%%%%%%%%%%%%%%%%
\subsection*{Acknowledgments}
The authors would like to thank Quanshui Wu for many useful conversations on
the subject and for sharing his ideas and proofs (see Section \ref{xxsec5})  
with them. The research of the first-named author was partially supported 
by the US National Security Agency (grant no.~H98230-14-1-0132), and that 
of the second-named author by the US National Science Foundation 
(grants no.~DMS 0855743 and DMS 1402863).

%%%%%%%%%%%%%%%%%%%%%%%%%%%%%%%%%%%%%%%%%
%%%%%%%%%%%%%%%%%%%%%%%
\section{Non-affine construction of types A and C}
\label{xxsec1}

We start by recalling a result of \cite{GZ} that classifies all
(not necessarily affine)
Hopf domains of GK-dimension one. Note that a domain of GK-dimension 
one is automatically commutative (e.g., \cite[Lemma 4.5]{GZ}).

\begin{lemma}
\label{xxlem1.1}
\cite[Proposition 2.1]{GZ}
Assume that a Hopf algebra $H$ is a domain of
GK-dimension one. Then $H$ is isomorphic to one of the following:
\begin{enumerate}
\item[(1)]
An enveloping algebra $U({\mathfrak g})$, where $\dim {\mathfrak g}= 1$.
\item[(2)]
A group algebra $k G$, where $G$ is infinite cyclic.
\item[(3)]
A group algebra $kG$, where $G$ is a non-cyclic torsionfree abelian group 
of rank $1$, i.e., a non-cyclic subgroup of $\QQ$.
\end{enumerate}
As a consequence, $H$ satisfies $(\natural)$.
\end{lemma}

\begin{proof}
The main assertion is \cite[Proposition 2.1]{GZ} and 
the consequence follows by an easy computation.
\end{proof}

Note that every Hopf algebra $H$ in Lemma \ref{xxlem1.1} is countable
dimensional and is completely determined by its coradical.

In \cite[Constructions 1.2, 1.3 and 1.4]{GZ}, we constructed some affine Hopf 
domains of GK-dimension two, labeled as types A, B and C. Non-affine
versions of types A and C can be constructed similarly and appeared also in 
other papers. One way of defining these is to use the Hopf Ore extensions 
introduced in \cite{Pa, BO}. We review the definition briefly, and refer
to \cite{Pa, BO} for more details.

Given a Hopf algebra $K$, an automorphism $\sigma$ and a $\sigma$-derivation
$\delta$ of $K$, a \emph{Hopf Ore extension} (or \emph{HOE}, for short)
of $K$, denoted by $K[z; \sigma,\delta]$, is a Hopf algebra $H$
that is isomorphic to the usual Ore extension $K[z; \sigma,\delta]$ 
as an algebra and 
contains $K$ as a Hopf subalgebra. HOEs have been studied in 
several papers including \cite{Pa, BO, WZZ3}. When $\delta=0$, the HOE 
$H$ is abbreviated to $K[z;\sigma]$; and when $\sigma=Id_K$, it is 
abbreviated to $K[z;\delta]$. If $K$ is a domain, then $H$ is 
also a domain. 

\begin{example}\cite[Example 5.4]{WZZ3}
\label{xxex1.2}
Let $K = kG$ where $G$ is a group and let $\chi: G \to k^{\times}$ be a 
character of $G$. Define an algebra automorphism $\sigma_{\chi}: K \to K$ 
by 
$$\sigma_{\chi}(g) = \chi(g)g, \quad \forall\; g\in G.$$
Let $\delta= 0$. By \cite[Example 5.4]{WZZ3}, $H := K[z;\sigma_{\chi}]$ is 
a HOE of $K$ with $\Delta(z)=z\otimes 1+e\otimes z$ for any choice of $e$ in 
the center of $G$. This Hopf algebra is denoted by $A_G(e,\chi)$.

We are mostly interested in nontrivial subgroups $G\subseteq \QQ$. In this 
case, 
$$\GKdim kG=1, \quad {\text{and}}\quad \GKdim H=2.$$ 
If $k={\mathbb C}$, then there are 
many characters of $G$. For example, let $\lambda$ be a real number; then 
$\exp_{\lambda}: r\to \exp(2\pi i r \lambda)$ is a  character from 
$\QQ$ to ${\mathbb C}$. 

A special case is when $G={\mathbb Z}$ (identified
with $\{x^{i}\}_{i\in {\mathbb Z}}$). Suppose the character $\chi: G\to 
k^\times$ is trivial (in this case $\sigma= \Id_{kG}$) and 
$\Delta(z)=z\otimes 1+x\otimes z$. This special HOE, denoted by 
$A_{{\mathbb Z}}(1,0)$, is the commutative Hopf algebra 
$A(1,1)$ given in \cite[Construction 1.2]{GZ} (by taking $(n,q)=(1,1)$).
More generally, if $n \in \ZZ$, $q\in \kx$, and $\chi: \ZZ \rightarrow \kx$ 
is the character given by $\chi(i) = q^{-i}$, then 
$A_\ZZ(n,\chi)$ is the Hopf algebra $A(n,q)$ of \cite[Construction 1.2]{GZ}.
\end{example}

\begin{example} \cite[A special case of Example 5.5]{WZZ3}
\label{xxex1.3} 
Let $K = kG$ where $G$ is a group and let $e$ be an element in the center
of $G$. Let $\tau: G \to (k,+)$ be an additive character of $G$.
Define a $k$-linear derivation $\delta : K \to K$ by 
$$\delta(g) = \tau(g)g(e-1),\quad \forall\;  g\in G.$$ 
Then $H := K[z;\delta]$ is a HOE of $K$ with $\Delta(z)=z\otimes 1+e\otimes z$.
This Hopf algebra is denoted by $C_G(e, \tau)$.

Later we will take $G$ to be a subgroup of $\QQ$. 
Since we assume $k$ has characteristic zero, 
there are many additive characters from $G\to (k,+)$. For example, 
let $\lambda$ be a rational number; then $i_{\lambda}: r\to r \lambda$ is 
an additive character from $G\to (k,+)$.

A special case is when $G = \ZZ$ (identified with $\{x^i\}_{i\in \ZZ}$), 
$e = 1-n$ for some $n \in \NN$, and $\tau = i_1$ (the inclusion map 
$\ZZ \rightarrow k$). Then $C_\ZZ(1-n, i_1)$ is the Hopf algebra $C(n)$ 
of  \cite[Construction 1.4]{GZ}.
\end{example}

The next result of \cite{WZZ3} says that $A_G(e,\chi)$ and 
$C_G(e, \tau)$ are natural classes of Hopf algebras. Let $G(H)$
denote the group of all grouplike elements in a Hopf algebra
$H$.

\begin{theorem}\cite[Theorem 7.1]{WZZ3}
\label{xxthm1.4}
Let $H$ be a pointed Hopf domain. Suppose that $G:=G(H)$ is abelian
and that
$$\GKdim kG < \GKdim H < \GKdim kG + 2 <\infty.$$
If $H$ does not contain $A(1, 1)$ as a Hopf subalgebra, 
then $H$ is isomorphic to either $A_G(e, \chi)$ 
or $C_G(e, \tau)$ as given in Examples {\rm{\ref{xxex1.2}-\ref{xxex1.3}}}. 
\end{theorem}

As a consequence,

\begin{corollary}
\label{xxcor1.5} 
Let $H$ be a pointed Hopf domain of GK-dimension two.
Suppose that the coradical of $H$ has GK-dimension one and 
that $H$ does not contain $A(1, 1)$ as a Hopf subalgebra.
Then $H$ is isomorphic to either $A_G(e, \chi)$ 
or $C_G(e, \tau)$ where $G$ is a nonzero subgroup of 
$\QQ$. 
\end{corollary}

\begin{proof} Since $H$ is pointed, the coradical of $H$
is $kG$ where $G=G(H)$. 
Since $\GKdim kG=1$, by Lemma \ref{xxlem1.1},
$G$ is a nonzero subgroup of $\QQ$. So $G(H)=G$ is
abelian. Now
$$1=\GKdim kG< 2=\GKdim H< \GKdim kG+2=3.$$
Hence, the hypothesis of Theorem \ref{xxthm1.4} holds, 
and the assertion follows from the theorem.
\end{proof}

Hopf algebras of GK-dimension two that contain $A(1,1)$
are more complicated. We will construct a family of 
them in the next section.

%%%%%%%%%%%%%%%%%%%%%%%%%%%%%%%%%%%%%%%%%
%%%%%%%%%%%%%%%%%%%%%%%
\section{Non-affine construction of type $B$}
\label{xxsec2}

%%%%%%%%%%%%%%%%%%%%%%%
\subsection{Construction}
\label{xxsec2.1}
We construct a Hopf algebra $B_G(\{p_i\},\chi)$ based 
on the following data.

{\bf Data.}
Let $G$ be a subgroup of $(\QQ,+)$ that contains $\ZZ$, 
and write 
its group algebra in the form
$$kG = \bigoplus_{a\in G} k\, x^a,$$
as in Notation \ref{xxsec0.1}. 

Let $I$ be an index set with $|I| \ge 2$. Let $\{ p_i \mid i\in I\}$ 
be a set of pairwise relatively prime integers such that 
$p_i\ge2$ and $1/p_i \in G$ for all $i\in I$, and let $M$ be 
the additive submonoid of $\QQ$ generated by $\{1/p_i \mid i\in I\}$. 
Due to the relative primeness assumption, $1/p_ip_j \in G$ for all 
distinct $i,j\in I$. Obviously $M\subseteq G$, but we want to keep 
the algebras of $M$ and $G$ separate, as these are playing 
different roles. Following Notation \ref{xxsec0.1}
we write the monoid algebra of $M$ in the form
$$kM = \bigoplus_{b\in M} k\, y^b.$$
Set 
$$GM = \{ a_1m_1+ \cdots+ a_tm_t \mid t\in\Znonneg, \; a_l\in G, 
\; m_l\in M\} = \sum_{i\in I} G(1/p_i),$$
and let $\chi: GM \rightarrow \kx$ be a character (i.e., a group 
homomorphism) such that 
\begin{enumerate}
\item[(1)]
$\chi(1/p_i^2)$ is a primitive $p_i$-th root of unity for all $i\in I$. 
\end{enumerate}
Note that (1) implies that 
\begin{enumerate}
\item[(2)]
$\chi(1/p_i) = \chi(1/p_i^2)^{p_i} = 1$ for $i \in I$, and
\item[(3)]
$\chi(1)=\chi(1/p_i)^{p_i}=1$.
\end{enumerate}

{\bf Observation.}
For any distinct $i,j\in I$, we have $1/p_ip_j \in GM$ and 
there are $c_i,c_j\in \ZZ$ such that $c_ip_i + c_jp_j =1$, whence
\begin{equation}  
\label{E2.0.1} \tag{2.0.1}
\chi(1/p_ip_j) = \chi((c_i/p_j)+(c_j/p_i)) 
= \chi(1/p_j)^{c_i} \chi(1/p_i)^{c_j} = 1.
\end{equation}

{\bf Algebra structure.}
Let $G$ act on $kM$ by $k$-algebra automorphisms such that
$$a \cdot y^b = \chi(ab) y^b \qquad \forall \; a\in G, \;\; b\in M.$$
Use this action to turn $kM$ into a left $kG$-module algebra, 
form the smash product 
$$B := kM \# kG,$$
and omit $\#$s from expressions in $B$. If we write 
$y_i := y^{1/p_i}$ for $i\in I$, then we can present $B$ by 
the generators $\{x^a \mid a\in G\} \sqcup \{y_i \mid i\in I\}$ 
and the relations
\begin{equation}  \label{E2.0.2} \tag{2.0.2}
\begin{aligned}
x^0 &= 1  \\
x^a x^{a'} &= x^{a+a'}  &&(a,a' \in G)  \\
x^a y_i &= \chi(a/p_i) y_i x^a  &\qquad&(a\in G, \;\; i \in I)  \\
y_i y_j &= y_j y_i  &&(i,j \in I)  \\
y_i^{p_i} &= y_j^{p_j}  &&(i,j \in I).
\end{aligned}
\end{equation}
These relations are very similar to the relations in 
\cite[(E1.2.1) in Construction 1.2]{GZ}. 

Consider a nonempty finite subset $J \subseteq I$. If $c$ is the 
product of the $p_j$ for $j \in J$, then the submonoid 
of $M$ generated by $\{ 1/p_j \mid j\in J\}$ is a submonoid of 
$M_c := \Znonneg (1/c)$. The subalgebra of $B$ generated by 
$\{x^a \mid a\in G\} \sqcup \{y_j \mid j\in J \}$ is a subalgebra 
of a skew polynomial ring
$$kM_c \# kG = kG[y^{1/c}; \sigma_c],$$
where $\sigma_c$ is the automorphism of $kG$ such that 
$\sigma_c(x^a) = \chi(-a/c) x^a$ for all $a\in G$. It follows 
that $kM_c \# kG$ is a domain of GK-dimension $2$. 

Since $B$ is a directed union of subalgebras of the form $kM_c \# kG$, we 
conclude that $B$ is a domain of GK-dimension $2$.

{\bf Hopf structure.}
It is clear from the presentation in \eqref{E2.0.2} that 
there is an algebra homomorphism $\eps: B\rightarrow k$ such 
that  $\eps(x^a)=1$ for all $a\in G$ and $\eps(y_i) =0$ for 
all $i\in I$.

Obviously $(x^a \otimes x^a)(x^{a'}\otimes x^{a'}) = x^{a+a'} 
\otimes x^{a+a'}$ for all $a,a' \in G$. Set $x_i := x^{1/p_i}$ 
and $\delta_i := y_i\otimes 1 + x_i \otimes y_i$ for $i\in I$. 
It is clear that $(x^a \otimes x^a) \delta_i = \chi(a/p_i) \delta_i 
(x^a \otimes x^a)$ for all $a\in G$ and $i\in I$. For any distinct 
$i,j\in I$, we have
\begin{equation}  \label{E2.0.3} \tag{2.0.3}
x_i y_j = \chi(1/p_ip_j) y_j x_i = y_j x_i
\end{equation}
because of \eqref{E2.0.1}, and likewise $x_j y_i = y_i x_j$. 
It follows that $\delta_i \delta_j = \delta_j \delta_i$. 
Moreover, since 
$$(x_i \otimes y_i) (y_i\otimes 1) = \chi(1/p_i^2) 
(y_i\otimes 1)(x_i \otimes y_i)$$
with $\chi(1/p_i^2)$ a primitive $p_i$-th root of unity, it 
follows from the $q$-binomial formula that 
$\delta_i^{p_i} = y^1\otimes 1 + x^1\otimes y^1$. Likewise, 
$\delta_j^{p_j} = y^1\otimes 1 + x^1\otimes y^1$, so that 
$\delta_i^{p_i} = \delta_j^{p_j}$. Therefore there is an 
algebra homomorphism $\Delta: B\rightarrow B\otimes B$ such 
that $\Delta(x^a) = x^a\otimes x^a$ for all $a\in G$ and 
$\Delta(y_i) = \delta_i$ for all $i\in I$.

Observe that $(\eps\otimes\id)\Delta(x^a) = x^a$ for all 
$a\in G$ and $(\eps\otimes\id)\Delta(y_i) = y_i$ for all 
$i\in I$. Consequently, $(\eps\otimes\id)\Delta = \id$, and 
similarly $(\id\otimes\,\eps)\Delta = \id$. We also observe 
that $(\Delta\otimes\id)\Delta$ and $(\id\otimes\, \Delta)\Delta$ 
agree on $x^a$ and $y_i$ for all $a\in G$ and $i\in I$, and 
consequently $(\Delta\otimes\id)\Delta = (\id\otimes\, \Delta)\Delta$. 
Therefore $(B,\Delta,\eps)$ is a bialgebra.

Next, observe that $(x^{-a'})(x^{-a}) = x^{-(a+a')}$ for all 
$a,a'\in G$ and that 
$$(-x_i^{-1} y_i) x^{-a} = \chi(a/p_i) x^{-a} (-x_i^{-1} y_i)$$
for all $a\in G$ and $i\in I$. For any distinct $i,j\in I$, we 
see using \eqref{E2.0.3} that
$$(-x_j^{-1} y_j)(-x_i^{-1} y_i) = (-x_i^{-1} y_i)(-x_j^{-1} y_j).$$
We also have
$$(-x_i^{-1} y_i)^{p_i} = (-1)^{p_i} x_i^{-1} y_i x_i^{-1} y_i 
\cdots x_i^{-1} y_i = (-1)^{p_i} \chi(-1/p_i^2)^{p_i(p_i+1)/2} 
y^1 x^{-1}.$$
If $p_i$ is odd, then $p_i$ divides $p_i(p_i+1)/2$ and so 
$\chi(-1/p_i^2)^{p_i(p_i+1)/2} = 1$. On the other hand, if 
$p_i$ is even, then $\chi(-1/p_i^2)^{p_i(p_i+1)/2} = (-1)^{p_i+1} = -1$ 
due to the primitivity of $\chi(-1/p_i^2)$. In both cases, we 
end up with $(-x_i^{-1} y_i)^{p_i} = -y^1x^{-1}$. Likewise, 
$(-x_j^{-1} y_j)^{p_j} = -y^1x^{-1}$, so that 
$(-x_i^{-1} y_i)^{p_i} = (-x_j^{-1} y_j)^{p_j}$. Therefore 
there is an algebra homomorphism $S: B\rightarrow B^{\op}$ 
such that $S(x^a) = x^{-a}$ for all $a\in G$ and $S(y_i) = 
-x_i^{-1} y_i$ for all $i\in I$.

Finally, observe that $m(S\otimes\id)\Delta(x^a) = 1 = \eps(x^a)1$ 
for all $a\in G$, where $m: B\otimes B\rightarrow B$ is the 
multiplication map, and $m(S\otimes\id)\Delta(y_i) = 0 = \eps(y_i)1$ 
for all $i\in I$, from which we conclude that $m(S\otimes\id)\Delta 
= u\circ\eps$, where $u: k\rightarrow B$ is the unit map. Similarly, 
$m(\id\otimes\, S)\Delta = u\circ\eps$. Therefore $(B,\Delta,\eps,S)$ 
is a Hopf algebra. We shall denote it $B_G(\{p_i\},\chi)$. The Hopf 
algebra structure is uniquely determined by the conditions
\begin{equation}  \label{E2.0.4} \tag{2.0.4}
\begin{aligned}
x^a\;\; &\text{is grouplike for all} \;\; a\in G,  \\
y_i \;\; &\text{is} \;\; (1,x_i)\text{-skew primitive for all} \;\; i\in I.
\end{aligned}
\end{equation}

The construction above can also be carried out when the index set $I$ is a singleton, but then the resulting Hopf algebra is isomorphic to $A_G(e,\chi)$ for suitable $e$ and $\chi$. We leave the details to the reader.

%%%%%%%%%%%%%%%%%%%%%%%
\subsection{Examples}
\label{xxsec2.2}
The data $(G,\{p_i\},\chi)$ can be chosen so that $B_G(\{p_i\},\chi)$ 
contains infinitely many skew primitive elements, as follows. 

\begin{example}
\label{xxex2.1}
Let $p_1,p_2,\dots$ be any infinite sequence of pairwise relatively 
prime integers $\ge2$, and let $G$ be the subgroup of $(\QQ,+)$  
generated by $1/p_i$ for all $i \in \NN$. Set 
$M := \sum_{i=1}^\infty \Znonneg (1/p_i)$ and $G^2 := 
\sum_{i=1}^\infty \ZZ(1/p_i^2)$. Note that $GM=G^2$, which is a 
subgroup of $(\QQ,+)$.
For each $i\in\NN$, let $\beta_i \in \kx$ be a primitive $p_i$-th root of 
unity. The cosets $\overline{1/p_i^2}$ in $\QQ/\ZZ$ generate finite 
cyclic subgroups of pairwise relatively prime orders, whence the 
sum of these subgroups is a direct sum. Hence, there is a homomorphism 
$$\overline{\chi}_0: \sum_{i=1}^\infty \ZZ 
\bigl(\, \overline{1/p_i^2} \,\bigr)=\bigoplus_{i=1}^\infty \ZZ 
\bigl(\, \overline{1/p_i^2} \,\bigr) \rightarrow \kx$$ 
such that $\overline{\chi}_0\bigl(\, \overline{1/p_i^2}\, \bigr) = \beta_i$ 
for all $i$. Since $\kx$ is a divisible abelian group, it is injective 
in the category of all abelian groups, and so $\overline{\chi}_0$ extends 
to a homomorphism 
$\overline{\chi}: G^2/\ZZ \rightarrow \kx$. Compose $\overline{\chi}$ 
with the quotient map $G^2\rightarrow G^2/\ZZ$ to obtain a character 
$\chi: G^2 \rightarrow \kx$. By the choice of $\chi$, we have 
$\chi(1/p_i^2) = \beta_i$ for all $i\in \NN$. As a consequence,
$\chi(a)=1$ for all $a\in G$. 

Thus, by the construction in the previous subsection, we obtain a
Hopf algebra $B_G(\{p_i\},\chi)$ that contains (infinitely many) 
distinct skew primitive elements $y_i$, for $i\in \NN$.

Since there are uncountably many different choices of 
$\overline{\chi}_0$, there are uncountably many 
non-isomorphic Hopf domains of GK-dimension two by 
Proposition \ref{xxpro2.5} below.
\end{example}

%\end{rmconstruction}

Certain natural finitely generated subalgebras of $B_G(\{p_i\},\chi)$ 
are Hopf algebras isomorphic to some of the Hopf algebras $A(n,q)$ 
and $B(n,p_0,\dots,p_s,q)$ of \cite[Constructions 1.1, 1.2]{GZ}, 
as follows.

\begin{lemma}  \label{xxlem2.2}
Let $B:= B_G(\{p_i\},\chi)$ as in the previous subsection.
Let $\Gtil$ be a finitely generated subgroup of $G$, $\Itil$ 
a nonempty finite subset of $I$ such that $1/p_i \in \Gtil$ 
for all $i\in \Itil$, and $\Btil$ the subalgebra of $B$ generated 
by $\{x^a \mid a\in \Gtil\} \sqcup \{y_i \mid i\in \Itil\}$. Then 
$\Btil$ is a Hopf subalgebra of $B$.

Assume that $\Itil = \{1,\dots,s\}$ for some positive integer 
$s$ and $p_1< \cdots< p_s$. Set $m:= p_1p_2 \cdots p_s$ and 
$m_i := m/p_i$ for $i\in \Itil$.
\begin{enumerate}
\item[(1)]
There are positive integers $n$ and $p_0$ such that 
$\Gtil = \ZZ(1/mn)$ and $1/m^2n \in GM$, while 
$q := \chi(1/m^2n)$ is a primitive $\ell$-th root of unity, 
where $\ell := mn/p_0$. Moreover, $p_0 \mid n$ and $p_0$ is 
relatively prime to each of $p_1,\dots,p_s$.
\item[(2)]
If $s=1$, then $\Btil \cong A(n,q)$.
\item[(3)]
If $s\ge 2$, then $\Btil \cong B(n,p_0,\dots,p_s,q)$.
\end{enumerate}
\end{lemma}

\begin{proof} (1)
Since $\Gtil$ is a finitely generated subgroup of $\QQ$ 
containing $\ZZ$, it has the form $\Gtil = \ZZ(1/t)$ for 
some positive integer $t$. For $i\in \Itil$, we have 
$1/p_i \in \Gtil$, whence $p_i \mid t$. Then, since the 
$p_i$ are pairwise relatively prime, $m\mid t$. Thus, 
$t = mn$ for some positive integer $n$. Let $\Mtil$ be
the submonoid $\sum_{i\in \Itil} \Znonneg (1/p_i)$. 

The pairwise relative primeness of the $p_i$ implies 
that $\gcd(m_1,\dots,m_s) = 1$, and so there exist 
integers $c_i$ such that $c_1m_1+ \cdots+ c_sm_s =1$, 
whence
\begin{equation}  \label{E2.2.1} \tag{2.2.1}
(c_1/p_1)+ \cdots+ (c_s/p_s) = 1/m .
\end{equation}
This does not imply that $1/m \in \Mtil$, since some of 
the $c_i$ may be negative, but we do get
$$\frac1{m^2n} = \frac{c_1}{p_1mn} + \cdots+ \frac{c_s}{p_smn} 
= \biggl( \frac{c_1}t \biggr) \left( \frac1{p_1} \right) 
+\cdots + \biggl( \frac{c_s}t \biggr) \biggl( \frac1{p_s} \biggr) 
\in \Gtil \Mtil\subseteq GM.$$
Therefore $q := \chi(1/m^2n) \in \kx$ is defined. Since 
$q^{m^2n} = \chi(1) = 1$, the order of $q$ in the group $\kx$ is 
finite, say $|q| = \ell$. Thus, $q$ is a primitive $\ell$-th root 
of unity. For $i\in \Itil$, the power
$$q^{m_i^2n} = \chi(1/m^2n)^{m_i^2n} = \chi(1/p_i^2)$$
is a primitive $p_i$-th root of unity, which implies that $p_i \mid \ell$. 
Consequently, $m\mid \ell$. On the other hand, $q^{p_im_i^2n} = 1$, 
whence $\ell$ divides $p_im_i^2n = m_imn$ for all $i\in \Itil$, and 
so $\ell \mid mn$. Thus, $mn = \ell p_0$ for some positive integer 
$p_0$. Since $m\mid \ell$, it follows that $p_0\mid n$.

Since $\chi(1/p_i) = 1$ for all $i\in \Itil$, we can invoke 
\eqref{E2.2.1} to obtain $\chi(1/m) = 1$, from which it follows 
that $\chi(1/m_i) =1$ for any $i\in \Itil$. Set $d_i := \gcd(p_0,p_i)$, 
and write $p_0= d_iu_i$ and $p_i = d_iv_i$ for some positive integers 
$u_i$, $v_i$. Since $\ell d_iu_i= \ell p_0= mn = d_iv_im_in$, we find 
that $\ell \mid v_im_in$, and consequently
$$1= q^{v_im_in} = \chi(1/m^2n)^{v_im_in} = \chi(1/p_i^2m_i)^{v_i}.$$
Now $p_i^2$ and $m_i$ are relatively prime, whence $a_ip_i^2+ b_im_i = 1$ 
for some integers $a_i$, $b_i$, and so $1/p_i^2m_i = (a_i/m_i)+ (b_i/p_i^2)$. 
Thus,
$$1 = \chi(p_i^2m_i)^{v_i} = \chi(a_i/m_i)^{v_i}  \chi(b_i/p_i^2)^{v_i} 
= \chi(1/p_i^2)^{b_iv_i}.$$
Since $\chi(1/p_i^2)$ is a primitive $p_i$-th root of unity, 
$p_i \mid b_iv_i$, from which it follows that $p_i$ divides 
$a_ip_i^2v_i+ b_im_iv_i = v_i$, and so $d_i=1$. Thus, $p_0$ and 
$p_i$ are relatively prime, for each $i\in \Itil$. By now, we have checked
all assertions in part (1).

(2) In this case, $m=p_1$. Set $\xtil := x^{1/mn}$, so that $\Btil$ 
is generated by $\{ \xtil^{\pm1}, y_1\}$. We have 
$\xtil y_1 = \chi(1/mnp_1) y_1\xtil = qy_1\xtil$, so there is an 
algebra isomorphism $\phi: \Btil \rightarrow A(n,q)$ with 
$\phi(\xtil) = x$ and $\phi(y_1) = y$. Since $x_1= x^{1/p_1} = \xtil^n$ 
and $\Delta(y_1) = y_1\otimes 1+ x_1\otimes y_1$, we see that $\phi$ 
preserves comultiplication. Observe also that $\phi$ preserves counit 
and antipode. Therefore $\phi$ is an isomorphism of Hopf algebras.

(3) Again, set $\xtil := x^{1/mn}$, and observe that $\Btil$ can be 
presented by the generators $\xtil^{\pm1},y_1,\dots,y_s$ and the relations
\begin{equation}  \label{E2.2.2} \tag{2.2.2}
\begin{aligned}
\xtil \xtil^{-1} &= \xtil^{-1} \xtil = 1   \\
\xtil y_i &= q^{m_i} y_i \xtil  &\qquad&(1\le i\le s)  \\
y_iy_j &= y_jy_i  &\qquad&(1\le i<j\le s)  \\
y_i^{p_i} &= y_j^{p_j}  &\qquad&(1\le i<j\le s).
\end{aligned}
\end{equation}
Comparing \eqref{E2.2.2} with \cite[(E1.2.1)]{GZ}, we see that there 
is an algebra isomorphism $\phi: \Btil \rightarrow B(n,p_0,\dots,p_s,q)$ 
such that $\phi(\xtil) = x$ and $\phi(y_i) = y_i$ for $i=1,\dots,s$. 
Since $\phi$ also preserves the Hopf algebra structures, we conclude 
that $\phi$ is an isomorphism of Hopf algebras.
\end{proof}

\begin{proposition}
\label{xxpro2.3}
Let $B:= B_G(\{p_i\},\chi)$ as in subsection {\rm\ref{xxsec2.1}}.
There is an ascending chain of Hopf subalgebras
$$B\langle 1\rangle \subseteq B\langle 2\rangle 
\subseteq \cdots \subseteq B\langle n\rangle \cdots \subseteq B$$
such that $B=\bigcup_{n=1}^{\infty} B\langle n\rangle $
and each $B\langle n\rangle$ is a finitely generated Hopf
algebra of type $B$ as in 
\cite[Construction 1.2]{GZ}.
\end{proposition}

\begin{proof} Since $G$ is countable, we list its elements as 
$\{g_1,g_2,\cdots, g_n,\cdots\}$. Write $I$ as either $\{1,2,\dots\}$ or $\{1,\dots,t\}$, and in the latter case set $p_i=p_t$ for all $i>t$. For $n\in \NN$, let $\Gtil\langle n\rangle$
be the subgroup of $G$ generated by 
$\{g_1,\cdots, g_n,1/p_1, \dots, 1/p_{n+1} \}$. Let $B\langle n\rangle$ be the
subalgebra of $B$ generated by $\Gtil\langle n \rangle$
and $\{y_1, \dots, y_{n+1} \}$. By Lemma \ref{xxlem2.2}(3),
$B\langle n\rangle$ is a finitely generated Hopf
algebra of type $B$ as in \cite[Construction 1.2]{GZ}.
It is clear that $B=\bigcup_{n=1}^{\infty} B\langle n\rangle$.
\end{proof}

%%%%%%%%%%%%%%%%%%%%%%%
\subsection{Basic properties}
\label{xxsec2.3}
We next derive some basic properties of the Hopf algebras $B_G(\{p_i\},\chi)$. 
A skew primitive
element of this Hopf algebra is called {\it non-trivial} if it is not
in $kG$. Recall the notation $x := x^1$ and $y := y^1$.

\begin{lemma}  \label{xxlem2.4}
Let $B:= B_G(\{p_i\},\chi)$ as in subsection {\rm\ref{xxsec2.1}}.
\begin{enumerate}
\item[(1)]
The only grouplike elements of $B$ are the $x^a$ for $a\in G$.
\item[(2)]
For $a\in G$, all $(1,x^a)$-skew primitive elements of 
$B$ are in $k(1-x^a)+ kM$.
\item[(3)]
Every non-trivial skew primitive element of $B$ is of the form 
$by_i +c(1-x_i)$ or $b y+ c(1-x)$ for some scalars $b,c\in k$.
\end{enumerate}
\end{lemma}

\begin{proof} It  is clear from the presentation in \eqref{E2.0.2} 
that the identity map on $kG$ extends to an algebra map 
$\pi: B \rightarrow kG$ such that $\pi(y_i) =0$ for all $i\in I$. 
We observe that $\pi$ is a Hopf algebra map, and thus that the map 
$\rho := (\id\otimes\,\pi)\Delta : B \rightarrow B\otimes kG$ makes 
$B$ into a right $kG$-comodule algebra. We claim that the subalgebra 
of $\rho$-coinvariants, $B^{\co\rho}$, equals $kM$. Since 
$\Delta(y_i) = y_i\otimes 1+ x_i\otimes y_i$ for $i\in I$, we see 
that each $y_i$ is a $\rho$-coinvariant, and consequently 
$kM \subseteq B^{\co\rho}$.

Consider a nonzero element $z\in B$, and write 
$z = \sum_{l=1}^m x^{a_l} z_l$ for some distinct elements 
$a_l \in G$ and some nonzero elements $z_l \in kM$. Then
\begin{equation}  \label{E2.4.1} \tag{2.4.1}
\rho(z) = \sum_{l=1}^m x^{a_l} z_l \otimes x^{a_l}.
\end{equation}
If $z$ is a $\rho$-coinvariant, we must have $m=1$ and 
$a_1=0$, whence $z=z_1\in kM$. Therefore $B^{\co\rho} = kM$, as 
claimed.

(1) Let $z\in B$ be grouplike, and write $z = \sum_{l=1}^m x^{a_l} z_l$ 
as above. Then 
$$\rho(z) = (\id\otimes\,\pi)(z\otimes z) = \sum_{l=1}^m x^{a_l} z_l \otimes \pi(z).$$
Comparing this with \eqref{E2.4.1}, we find that $x^{a_l} = \pi(z)$ for all $l$. 
In particular, the elements $x^{a_l}$ are all the same, so we must have $m=1$, 
and $z= x^{a_1}z_1$. Now $z_1 = x^{-a_1}z$ is grouplike, and it suffices to 
show that $z_1=x^a$ for some $a\in G$. Thus, we may assume that $z\in kM$. 

Recall from subsection \ref{xxsec2.1} 
that $B$ is a directed union of subalgebras of skew polynomial rings of the 
form $kM_c \# kG = kG[y^{1/c}; \sigma_c]$ (see the end of {\bf Algebra structure}). 
In such a skew polynomial ring, the 
only units are the units of $kG$, so the only units in $B$ are those in $kG$. 
Since grouplike elements are units, we obtain $z\in kG$, and therefore $z=1=x^0$.

(2) We first show that any nonzero $(1,x^a)$-skew primitive element 
$w\in kG$ must be a scalar multiple of $1-x^a$. Write 
$w= \sum_{l=1}^m \alpha_l x^{a_l}$ for some $\alpha_l \in \kx$ and 
some distinct $a_l \in G$. Then
\begin{equation}  \label{E2.4.2} \tag{2.4.2}
\sum_{l=1}^m \alpha_l x^{a_l} \otimes x^{a_l} = \Delta(w) 
= \sum_{l=1}^m \alpha_l \bigl( x^{a_l} \otimes 1 
+ x^a \otimes x^{a_l} \bigr).
\end{equation}
It follows that any nonzero $a_l$ must equal $a$, whence $m\le 2$. 
If $m=1$, then, after multiplying by $\alpha_1^{-1}$, \eqref{E2.4.2} 
reduces to 
$$x^{a_1} \otimes x^{a_1} = x^{a_1}\otimes1 + x^a\otimes x^{a_1},$$
which is impossible. Thus, after a possible renumbering, we must 
have $a\ne0$ and $w= \alpha_1 + \alpha_2 x^a$. In this case, 
\eqref{E2.4.2} says that
$$\alpha_1\otimes1 + \alpha_2 x^a\otimes x^a = 
\alpha_1(1\otimes1+ x^a\otimes1) + 
\alpha_2(x^a\otimes 1+ x^a\otimes x^a).$$
It follows that $\alpha_1+\alpha_2=0$ and so $w= \alpha_1(1-x^a)$, 
as desired.

Now suppose that $z$ is a $(1,x^a)$-skew primitive element of $B$, 
for some $a\in G$. Then $\pi(z)$ must be a $(1,x^a)$-skew primitive 
element of $kG$, and so the claim above shows that 
$\pi(z) = \alpha (1-x^a)$ for some $\alpha \in k$. Consequently, 
$z' := z- \alpha (1-x^a)$ is a $(1,x^a)$-skew primitive element 
of $B$ with $\pi(z') =0$. Then
$$\rho(z') = (\id\otimes\,\pi)(z'\otimes 1 + x^a\otimes z') = z'\otimes 1,$$
whence $z'\in B^{\co\rho} = kM$. Therefore $z\in k(1-x^a)+ kM$.

(3) By Proposition \ref{xxpro2.3}, we may assume that $B$ is finitely
generated and isomorphic to a Hopf algebra of type B as in 
\cite[Construction 1.2]{GZ}. By \cite{WZZ1}, these type B Hopf
algebras form a special class of the $K(\{p_s\}, \{q_s\}, \{\alpha_s\}, M)$
defined in \cite[Section 2]{WZZ1}. By \cite[Lemma 2.9(a)]{WZZ1},
any non-trivial skew primitive element $f$ in $B$ is a linear
combination of $\{y_i\}_{i\in I}$ and $y$ modulo $kG$. Write 
a $(1,g)$-skew primitive element $f$ as 
$f=ay+ \sum_{i\in I} a_i y_i +f_0$ where $f_0\in kG$. By 
\cite[Lemma 2.9(a)]{WZZ1}, $g=x_i$ or $x$. Since all $x_i$ and $x$ are
distinct, we have that only one of $\{a_i\}_{i\in I}\cup \{a\}$ is 
nonzero. The assertion follows by combining this with part (2).
\end{proof}

The next proposition is similar to \cite[Lemma 1.3]{GZ}.

\begin{proposition}  
\label{xxpro2.5}
Let $B:= B_G(\{p_i\},\chi)$ and $\Btil:= B(\Gtil,\{\ptil_i\},\chitil)$ 
as in subsection {\rm\ref{xxsec2.1}}. Then 
$B \cong \Btil$ if and only if $G=\Gtil$, $\{p_i\mid i\in I\} = 
\{\ptil_i \mid i\in \Itil\}$, and $\chi= \chitil$.
\end{proposition}

\begin{proof} Let $\phi: B\rightarrow \Btil$ be an isomorphism of Hopf algebras.

Label the canonical generators of $B$ as above, namely as $x^a$ 
for $a\in G$ and $y_i$ for $i\in I$, and label those of $\Btil$ as 
$\xtil^a$ for $a\in \Gtil$ and $\ytil_i$ for $i\in \Itil$. Write 
$M$ and $\Mtil$ for the additive submonoids of $\QQ$ generated by 
$\{1/p_i \mid i\in I\}$ and $\{1/\ptil_i \mid i\in \Itil\}$, respectively.

The group of grouplike elements of $B$ is isomorphic to $G$, and 
that of $\Btil$ to $\Gtil$, so it follows from Lemma 
\ref{xxlem2.4}(1) that there is an isomorphism 
$\gamma: G\rightarrow \Gtil$ such that $\phi(x^a) = \xtil^{\gamma(a)}$ 
for all $a\in G$. Since $G$ and $\Gtil$ are additive 
subgroups of $\QQ$, $\gamma$ is given by multiplication by 
some $r\in \QQ^\times$. Thus, $\Gtil = rG$ and 
$\phi(x^a) = \xtil^{ra}$ for all $a\in G$.

We next show that $\phi$ maps $kM$ onto $k\Mtil$. For $i\in I$, 
the element $y_i$ is $(1,x^{1/p_i})$-skew primitive, and 
$x^{1/p_i}y_i= q_iy_i x^{1/p_i}$ where $q_i := \chi(1/p_i^2)$ 
is a primitive $p_i$-th root of unity. Then $\phi(y_i)$ is a 
$(1,\xtil^{r/p_i})$-skew primitive element of $\Btil$ such 
that $\xtil^{r/p_i} \phi(y_i) = q_i \phi(y_i) \xtil^{r/p_i}$. 
By Lemma \ref{xxlem2.4}(2), $\phi(y_i) = \alpha(1- \xtil^{r/p_i})+z$ 
for some $\alpha\in k$ and $z\in k\Mtil$. Hence,
$$q_i  \alpha(1- \xtil^{r/p_i})+ q_iz = q_i\phi(y_i) 
= \xtil^{r/p_i} \phi(y_i) \xtil^{-r/p_i} 
=  \alpha(1- \xtil^{r/p_i})+ \xtil^{r/p_i} z \xtil^{-r/p_i},$$
and so $(q_i-1) \alpha(1- \xtil^{r/p_i}) \in k\Mtil$. 
Since $q_i\ne 1$, this forces $\alpha=0$, whence 
$\phi(y_i) =z \in k\Mtil$. Thus, $\phi(kM) \subseteq k\Mtil$. 
By symmetry, $\phi^{-1}(k\Mtil) \subseteq kM$, and therefore 
$\phi(kM) = k\Mtil$.

As in the proof of Lemma \ref{xxlem2.4}, the identity map 
on $kG$ extends to a Hopf algebra map $\pi: B\rightarrow kG$ 
such that $\pi(y_i)=0$ for all $i\in I$, and there is a 
corresponding Hopf algebra map $\pitil: \Btil \rightarrow k\Gtil$. 
Since $\phi$ maps $kM\cap \ker\eps$ to $k\Mtil\cap \ker\eps 
\subseteq \ker\pitil$, we conclude that $\phi|_{kG} \pi = \pitil \phi$. 
Now the maps 
$$\lambda:= (\pi\otimes\id)\Delta : B\rightarrow kG\otimes B 
\qquad\text{and}\qquad \lambdatil:= (\pitil\otimes\id)\Delta : 
\Btil\rightarrow k\Gtil\otimes \Btil$$
make $B$ and $\Btil$ into left comodule algebras over $kG$ 
and $k\Gtil$, respectively, whence $B$ is $G$-graded and 
$\Btil$ is $\Gtil$-graded. Since $\phi|_{kG} \pi = \pitil \phi$, 
we see that $\bigl( \phi|_{kG}\otimes\phi \bigr) \lambda = 
\lambdatil \phi$, and thus $\phi$ transports the grading on $B$ 
to the grading on $\Btil$; namely, $\phi(B_a) = \Btil_{ra}$ 
for all $a\in G$.

For $a\in G$, we observe that $kM\cap B_a$ is nonzero if and 
only if $a\in M$, and similarly in $\Btil$. It follows that 
$rM = \Mtil$. Note that this forces $r>0$.

The atoms of the monoid $M$ (i.e., the additively 
indecomposable elements) are exactly the $1/p_i$ for $i\in I$, 
as one sees from the pairwise relative primeness of the $p_i$. 
Similarly, the atoms of $\Mtil$ are exactly the $1/\ptil_i$. 
Since we have an isomorphism $b\mapsto rb$ from $M$ onto 
$\Mtil$, it follows that $\{1/\ptil_i \mid i\in \Itil\} 
= \{r/p_i \mid i\in I\}$. Consequently, we may assume that 
$\Itil=I$ and $1/\ptil_i = r/p_i$ for all $i\in I$.

Finally, write $r=s/t$ for some relatively prime positive 
integers $s$, $t$, and reduce the final equation of the 
previous paragraph to $tp_i = s\ptil_i$. Thus, $s$ divides 
$tp_i$ for all $i\in I$. Since there are distinct indices 
$i,j\in I$, and $p_i$, $p_j$ are relatively prime, it follows 
that $s\mid t$. By symmetry, $t\mid s$, whence $r=1$. 
Therefore $\Gtil = G$ and $\ptil_i = p_i$ for all $i\in I$. 
Moreover, $\phi(kM \cap B_{1/p_i}) = k\Mtil\cap \Btil_{1/p_i}$, 
from which we see that $\ytil_i$ is a nonzero scalar multiple 
of $\phi(y_i)$ (due to the fact that $k\Mtil \cap \Btil_a$ is 
$1$-dimensional for all $a\in \Gtil$). It thus follows from 
the relations $x^a y_i = \chi(a/p_i) y_i x^a$ that 
$\xtil^a \ytil_i = \chi(a/p_i) \ytil_i \xtil^a$, whence 
$\chitil(a/\ptil_i) = \chi(a/p_i)$ for all $a\in G$ and $i\in I$. 
Therefore $\chitil = \chi$.
\end{proof}

%%%%%%%%%%%%%%%%%%%%%%%%%%%%%%%%%%%%%%%%%
%%%%%%%%%%%%%%%%%%%%%%%
\section{Initial analysis}
\label{xxsec3}

In this section we will finish most of the analysis of 
the pointed case. 

%%%%%%%%%%%%%%%%%%%%%%%
\subsection{Classification by GK-dimension of the coradical}
\label{xxec3.1}
Throughout this subsection we assume that Hopf algebras are 
pointed. We will review some results from other papers.

Suppose $H$ is a Hopf domain of GK-dimension two. Since
$H$ is pointed, the coradical $C_0(H)$ of $H$ is a group
algebra $kG$ for the group $G= G(H)$ of all group\-likes in $H$.
Hence, $\GKdim kG\leq \GKdim H=2$. Since $\GKdim kG$
is an integer, $\GKdim kG$ is either $0$, or $1$, 
or $2$, see \cite[Section 2.1]{WZZ1}. We shall refer to 
$\GKdim kG$ as the \emph{GK-dimension of $G$}, for short.

If $\GKdim kG=0$, then $C_0(H)=k$ (as $H$ is a domain).
This means that $H$ is connected. By \cite[Theorem 1.9]{WZZ1},
$H$ is isomorphic to $U({\mathfrak g})$ for a 2-dimensional
Lie algebra ${\mathfrak g}$. This is part (1) of the 
following proposition. 

For the statement of part (3) in the next proposition, 
recall that the localization $\ZZ_{(2)}$ is the subring 
of $\QQ$ consisting of all rational numbers with odd 
denominators. There is a nontrivial group homomorphism 
$\varphi : \ZZ_{(2)} \rightarrow \Aut \QQ$ such that 
$\ker\varphi = 2 \ZZ_{(2)}$ and the remaining elements 
of $\ZZ_{(2)}$ are sent to the automorphism $(-1)\cdot(-)$. 
We shall also use $\varphi$ to denote the corresponding 
homomorphism from a subgroup $R$ of $\ZZ_{(2)}$ to the 
automorphism group of a subgroup $L$ of $\QQ$.

\begin{proposition}
\label{xxpro3.1} Let $H$ be a pointed Hopf domain of GK-dimension 
two. 
\begin{enumerate}
\item[(1)] 
If $\GKdim C_0(H)=0$, then $H\cong U({\mathfrak g})$ for a 
$2$-dimensional Lie algebra ${\mathfrak g}$.
\item[(2)] 
If $\GKdim C_0(H)=2$ and $C_0(H)$ is commutative, then 
$H\cong kG$ where $G$ is a subgroup of 
${\mathbb Q}^2$ containing ${\mathbb Z}^2$.
\item[(3)]
If $\GKdim C_0(H)=2$ and $C_0(H)$ is not commutative, then 
$H\cong kG$ where $G = L\rtimes_\varphi R$ for some subgroup $L$ of 
${\mathbb Q}$ containing ${\mathbb Z}$
and some subgroup $R$ of $\ZZ_{(2)}$ containing $\ZZ$.
\end{enumerate}
\end{proposition}

\begin{proof} (1) This is \cite[Theorem 1.9]{WZZ1}.
 
Let $C_0(H)=kG$. In both (2) and (3), we 
have $\GKdim kG=2$. By \cite[Lemma 1.6]{WZZ1}, $H=C_0(H)=kG$. Then 
\begin{equation}
\label{E3.1.1} \tag{3.1.1}
G=\bigcup_{N\in \cN} N,
\end{equation}
a directed union, where $\cN$ is the set of all finitely 
generated subgroups of $G$ of GK-dimension two. 

(2) If $H=C_0(H)$ is commutative, meaning $G$ is abelian, then 
every $N \in \cN$ is isomorphic
to ${\mathbb Z}^2$ by \cite[Theorem 1.7]{WZZ1}. If $M\subseteq N
\subseteq G$ where $M$ has GK-dimension two, then $M\cong N\cong
{\mathbb Z}^2$ and $N/M$ is finite.
This is true for all such $N$, which implies that $G/M$ is torsion.
Therefore $G$ is isomorphic to a subgroup of ${\mathbb Q}^2$ containing 
${\mathbb Z}^2$. Conversely, every subgroup of 
${\mathbb Q}^2$ containing ${\mathbb Z}^2$ has GK-dimension two.
The assertion follows. 

(3) This is the case when $G$ is non-abelian. By \cite[Theorem 1.7]{WZZ1}, 
we may assume that every $N$ in $\cN$ is isomorphic to the nontrivial 
semidirect product ${\mathbb Z}\rtimes {\mathbb Z}$. Then $N$ is 
generated by elements  $x_N$ and $y_N$ satisfying $x_N y_N x_N^{-1} 
= y_N^{-1}$. As is well known, $x_N$ and $y_N$ have infinite order. 
Moreover, the following properties are easily checked:
\begin{enumerate}
\item[(a)] 
$Z(N) = \langle x_N^2 \rangle$.
\item[(b)] 
$Y_N := \langle y_N \rangle$ is a normal subgroup of $N$.
\item[(c)] 
$C_N := C_N(Y_N) = Y_N Z(N)$ and $[N:C_N] = 2$.
\item[(d)] 
For any $a\in N \setminus C_N$, we have $a^2 \in Z(N)$ and 
$ay_Na^{-1} = y_N^{-1}$. Moreover, $\langle a \rangle Z(N)$ is 
infinite cyclic.
\end{enumerate}
A short calculation reveals that, for a fix $a$ in (d),
$$\{ b\in N \mid aba^{-1} = b^{-1} \;\; 
\text{for some} \;\; a \in N \} = Y_N$$
for any $N \in \cN$. It follows that $Y_N \subseteq Y_{M}$ 
whenever $N \subseteq M$ in $\cN$, and so
$$Y_G := \bigcup_{N\in \cN} Y_N$$
is a normal subgroup of $G$. Since each $Y_N$ is infinite cyclic, 
$Y_G$ is isomorphic to a subgroup $L$ of $\QQ$ containing $\ZZ$.

If $N \subseteq M$ in $\cN$, then $N\cap C_{M} \subseteq C_N$, 
because $Y_N \subseteq Y_M$. Since $[M : C_M] = 2$, it follows 
that $N\cap C_M = C_N$. In particular, $C_N \subseteq C_M$. Now
$$C_G := \bigcup_{N \in \cN} C_N$$
is a normal subgroup of $G$ containing $Y_G$. Pick some $N_0 \in 
\cN$, and set
$$x_G := x_{N_0} \,.$$
For any $N \in \cN$ containing $N_0$, the equation $N_0 \cap C_N 
= C_{N_0}$ implies $x_G \notin C_N$. Consequently, in view of (c),
we have 
\begin{enumerate}
\item[(e)] $x_G \notin C_G$ and $[G:C_G] = 2$.
\end{enumerate}

When $N \subseteq M$ in $\cN$, we have $x_N \notin C_M$ 
because $N\cap C_M = C_N$, and so it follows from (d) and (a) 
that $Z(N) \subseteq Z(M)$. Thus,
$$Z(G) = \bigcup_{N \in \cN} Z(N).$$
If we now set
$$X_G := \langle x_G\rangle Z(G) = 
\bigcup_{N_0 \subseteq N \in \cN} \langle x_G \rangle Z(N),$$
then (d) tells us that $X_G$ is a directed union of infinite 
cyclic groups. Hence, $X_G$ is isomorphic to a subgroup $R$ of 
$\QQ$ containing $\ZZ$, with $x_G \mapsto 1$. Moreover,
$$X_G Y_G = \langle x_G \rangle C_G = G,$$
because of (e). Note also that $Y_G Z(G) \subseteq C_G$, 
whence $X_G \cap Y_G \subseteq \langle x_G^2 \rangle Z(G) = Z(G)$, 
and so $X_G \cap Y_G \subseteq Y_G \cap Z(G) = 1$. Therefore
$$G \cong Y_G \rtimes X_G \cong L \rtimes_\alpha R$$
for some homomorphism $\alpha : R \rightarrow \Aut L$.

Because of (e) and (d), we have $x_G y x_G^{-1} = y^{-1}$ 
for all $y \in Y_G$, so $\alpha(1)$ must be the automorphism 
$\nu := (-1)\cdot(-)$ of $L$. Since the automorphisms of $L$ 
are given by multiplication by certain elements of $\QQ^\times$, 
the automorphism $\nu$ has no square root in $\Aut L$, whence 
$1 \notin 2R$. It follows that $a/b \notin R$ for any odd 
integer $a$ and any nonzero even integer $b$, and therefore 
$R \subseteq Z_{(2)}$. We similarly conclude that 
$\alpha(R) = \{\id_L, \nu\}$, whence 
$\alpha(R \cap 2\ZZ_{(2)}) = \{\id_L\}$ and 
$\alpha(R \setminus 2\ZZ_{(2)}) = \{ \nu \}$. Therefore 
$\alpha = \varphi$, completing the proof.
\end{proof}

The only case left is when $C_0(H)=kG$ and $G$ has rank one.
By Lemma \ref{xxlem1.1}(3), $G$ is isomorphic to a subgroup of 
$\QQ$ containing ${\mathbb Z}$. Further analysis
assuming this is given in the next subsection. 

%%%%%%%%%%%%%%%%%%%%%%%
\subsection{Analysis of skew primitives}
\label{xxsec3.2}

In the first result of this subsection, we assume that $H$ is a 
pointed Hopf domain of GK-dimension two and that $C_0(H)=kG$ where
$G$ is a subgroup of $\QQ$ containing ${\mathbb Z}$. We identify 
elements $a\in G \subseteq \QQ$ with elements labeled $x^a \in G(H)$. 
Since type A and type C Hopf algebras are easy to understand 
(Examples \ref{xxex1.2} and \ref{xxex1.3}), we are focusing on 
algebras that are not types A and C. By Corollary \ref{xxcor1.5}, we 
may assume that $H$ contains $A(1,1)$ as a Hopf subalgebra. This means 
that $H$ contains a grouplike element $x\in G$ and a skew primitive
element $y\notin kG$ such that
\begin{equation}
\label{E3.1.2} \tag{3.1.2}
xy=yx, \quad {\text{and}}\quad \Delta(y)=y\otimes 1+ x\otimes y.
\end{equation}
By replacing $G$ with an isomorphic subgroup of $\QQ$, we may assume 
that $x = x^1$. 

\begin{lemma}
\label{xxlem3.2} If $y$ is, up to a scalar, the only nontrivial 
skew primitive element {\rm{(}}modulo $kG(H)${\rm{)}} in $H$, 
then $H$ is either type $A$ or type $C$.
\end{lemma}

\begin{proof}
This follows from the proof of \cite[Theorem 7.1]{WZZ3}.

Let $K$ be the subalgebra of $H$ generated by $y$ and $C_0(H)$. It 
is clear that $K$ is a Hopf subalgebra. Applying \cite[Lemma 2.2(c)]{WZZ2} 
to $V := ky+k (x-1)$, there is an element $z\in V\setminus k(x-1)$
such that either
\begin{enumerate}
\item[(i)]
there is a character $\chi: G\to k^\times$ 
such that $h^{-1} z h=\chi(h) z$ for all $h\in G$, or
\item[(ii)]
there is an additive character $\tau: G\to k$ 
such that $h^{-1} z h= z+ \tau(h) (x-1)$ for all $h\in G$.
\end{enumerate}
In the first case $K$ is a quotient of $A_G(x,\chi)$ and in 
the second case $K$ is a quotient of $C_G(x, \tau)$. We claim 
that $K \cong A_G(x,\chi)$ in the first case and that 
$K \cong C_G(x,\tau)$ in the second case. We only prove the 
claim for the second case (the first case was given in the 
proof of \cite[Theorem 7.1]{WZZ3}).
Consider the natural Hopf map $ f: C_G(x,\tau)\to H$
which is injective on $C_0+C_0 z= C_1(H)$ by definition. 
By \cite[Theorem 5.3.1]{Mo}, $f$ is injective. 
Consequently, $K \cong C_G(x,\tau)$. By definition, $K$ is 
generated by all the grouplikes and skew primitive 
elements of $H$. By \cite[Corollary 6.9(2)]{WZZ3}, 
${\text{PCdim}} K = 1$. Finally, by 
\cite[Proposition 2.4(2)]{WZZ3}, $H = K$ as desired.
\end{proof}

\begin{notation}  
\label{xxnot3.3}
We shall also need information about the skew primitive 
elements of the affine Hopf domains of types A, B, C 
from \cite[Section 1]{GZ}. To make the notation compatible 
with the present paper, we express these Hopf algebras as follows.
\begin{enumerate}
\item[(1)]
$A(n,q)$, for $n \in \Znonneg$ and $q \in \kx$, is presented by 
generators $x^{\pm1}$, $z$ with $xz= qzx$, where $x$ is 
grouplike and $z$ is $(1,x^n)$-skew primitive. We restrict to 
$n \ge 0$ because $A(m,q) \cong A(-m,q^{-1})$.
\item[(2)]
$B(n,p_0,\dots,p_s,q)$, for $s\in \ZZ_{\ge2}$, 
$n,p_0,\dots,p_s \in \Zpos$, and $q\in \kx$ satisfying certain 
conditions, is presented by generators $x^{\pm1},y_1,\dots,y_s$ 
with relations described in \cite[Eqn.~(E1.2.1)]{GZ}, where 
$x$ is grouplike and each $y_i$ is $(1,x^{m_in})$-skew primitive. 
(Here $m_i = m/p_i$ with $m = p_1p_2 \cdots p_s$.) The restriction 
$s\ge2$ rules out the situation $B(n,p_0,p_1,q) \cong A(n,q)$.
\item[(3)]
$C(n)$, for $n \in \ZZ_{\ge2}$, is presented by generators 
$x^{\pm1}$, $z$ with the relation $zx = xz + (x^{2-n} - x)$, where $x$ is 
grouplike and $z$ is $(1,x^{1-n})$-skew primitive. We 
restrict to $n \ge 2$ because $C(1) \cong A(0,1)$ and 
$C(m) \cong C(2-m)$. In particular, this means $zx \ne xz$.
\end{enumerate}
\end{notation}

\begin{proposition}  
\label{xxpro3.4}
\begin{enumerate}
\item[(1)] 
Let $A = A(n,q)$ where $q^n$ is either $1$ or a non-root of unity. 
Then all nontrivial skew primitive elements in $A$ have weight 
$x^n$, and they are linear combinations of $z$ and $1-x^n$.
\item[(2)]  
Let $A = A(n,q)$ where $q^n$ is a primitive $d$-th root of unity 
for some $d>1$. Then any nontrivial skew primitive element of 
$A$ has weight either $x^n$ or $x^{dn}$, and it is a linear
combination of $z$ and $1-x^n$ or of $z^d$ and $1-x^{dn}$, respectively.
\item[(3)] 
Let $B = B(n,p_0,\dots,p_s,q)$, with $m$, $m_i$ as above. 
Any nontrivial skew primitive element of $B$ has weight 
either $x^{mn}$ or $x^{m_in}$ for some $i=1,\dots,s$, and 
it is a linear combination of $y_1^{p_1}$ and $1-x^{mn}$ or 
of $y_i$ and $1-x^{m_in}$, respectively.
\item[(4)] Let $C = C(n)$. All nontrivial skew primitive 
elements in $C$ have weight $x^{1-n}$, and they are linear 
combinations of $z$ and $1-x^{1-n}$.
\end{enumerate}
\end{proposition}

\begin{proof} In each case, the elements of the stated forms are 
skew primitive with the given weights, as proved in 
\cite[Constructions 1.1, 1.2, 1.4]{GZ}. 

(1,2) Suppose $f \in A$ is a nontrivial skew primitive with 
weight $x^m$, for some $m\in\ZZ$. Write 
$f = \sum_{i=0}^d c_iz^i$ for some $c_i \in k[x^{\pm1}]$ with 
$c_d \ne 0$. Then  
\begin{equation}
\label{E3.4.1}\tag{3.4.1}
\sum_{i=0}^d \Delta(c_i) (z\otimes 1+ x^n\otimes z)^i 
= \Delta(f) 
= \sum_{i=0}^d c_iz^i\otimes 1 + \sum_{i=0}^d x^m \otimes c_iz^i.
\end{equation}
Comparing terms from $A \otimes k[x^{\pm1}]$ in this equation, 
we see that
$$
\sum_{i=0}^d \Delta(c_i) (z^i\otimes 1) 
= x^m\otimes c_0 + \sum_{i=0}^d c_iz^i \otimes 1.
$$
A comparison of terms from $k[x^{\pm1}] z^i \otimes k[x^{\pm1}]$ 
then yields $\Delta(c_i) = c_i \otimes 1$ for $i>0$ and 
$\Delta(c_0) = x^m\otimes c_0 + c_0\otimes 1$. It follows 
that $c_0 \in k(1-x^m)$ and $c_i \in k$ for $i>0$. The 
nontriviality of $f$ forces $d>0$.

If $d=1$, \eqref{E3.4.1} implies that $c_1 x^n \otimes z 
= x^m\otimes c_1z$, whence $m=n$. In this case, $f$ is a 
linear combination of $z$ and $1-x^n$, and we are done. Assume now that $d>1$.

Since $c_0 \in k(1-x^m)$ and $c_1,\dots,c_d$ are scalars, 
\eqref{E3.4.1} reduces to
$$
\sum_{i=1}^d c_i  (z\otimes 1+ x^n\otimes z)^i 
=  \sum_{i=1}^d  c_iz^i\otimes 1 +   \sum_{i=1}^d c_ix^m \otimes z^i .
$$
Comparing terms from $A \otimes kz^i$ yields
\begin{equation}
\label{E3.4.2}\tag{3.4.2}
\sum_{i=j}^d {i \choose j}_{q^n} c_i z^{i-j} x^{jn} 
= c_j x^m, \qquad \text{for} \;\; 1 \le j \le d.
\end{equation}
From the case $j=d$, we get $m= dn$. 

For $1\le j<d$, \eqref{E3.4.2} implies that 
${d \choose j}_{q^n} = 0$. Thus, by \cite[Lemma 7.5]{GZ}, 
$q^n$ must be a primitive $d$-th root of unity. As in 
\cite[Construction 1.1]{GZ}, it follows that $z^d$ is skew 
primitive with weight $x^m$, hence so is $f - c_d z^d$. If 
$f - c_d z^d$ is nonzero, let $e$ be its $z$-degree. Applying 
the above analysis to $f - c_d z^d$, we find that $e = 0$ or 
$m = en$, the latter case being impossible. Therefore 
$f - c_d z^d$ is a scalar multiple of $1 - x^{dn}$, and $f$ 
has the required form.

(3) This follows from Lemma \ref{xxlem2.4}.

(4) Suppose $f \in C$ is a nontrivial skew primitive with 
weight $x^m$, for some $m\in\ZZ$, and write  
$f = \sum_{i=0}^d c_iz^i$ for some $c_i \in k[x^{\pm1}]$ 
with $c_d \ne 0$. As in cases (1)(2), we get equation 
\eqref{E3.4.1}, but with $x^n$ replaced by $x^{1-n}$. 
Moreover, it follows that $c_0 \in k(1-x^m)$ and 
$c_i \in k$ for $i>0$, and then that $d>0$. In case $d=1$, 
we obtain $m= 1-n$ and $f$ is a linear combination of $z$ 
and $1-x^{1-n}$.

Now suppose that $d>1$. After canceling common terms, 
\eqref{E3.4.1} reduces to
$$
\sum_{i=1}^d c_i  (z\otimes 1+ x^{1-n}\otimes z)^i 
=  \sum_{i=1}^d  c_iz^i\otimes 1 +   \sum_{i=1}^d c_ix^m \otimes z^i .
$$
Comparing terms in $C \otimes k[x^{\pm1}] z^d$ in this 
equation, we find that $m=d(1-n)$. Turning to  
$C \otimes k[x^{\pm1}] z^{d-1}$, we obtain
$$
c_{d-1} x^{(d-1)(1-n)} + d c_d x^{(d-1)(1-n)} z + c_d h = c_{d-1} x^{1-n} 
$$
for some $h \in k[x^{\pm1}]$. Since $d c_d \ne 0$, this 
is impossible, and the proof is complete.
\end{proof}

\begin{corollary}
\label{xxcor3.5}
Suppose $H$ is one of the affine Hopf domains of 
types A, B, C. For any grouplike $g$ in $H$, the space of 
skew primitive elements in $H$ with weight $g$ has 
$k$-dimension at most $2$.
\end{corollary}

\begin{corollary}
\label{xxcor3.6}
Let $H_1 \subsetneqq H_2$ be affine Hopf domains of types A, B, C.
\begin{enumerate}
\item[(1)] 
If $H_2$ is of type A, so is $H_1$.
\item[(2)] 
If $H_1$ is of type B, so is $H_2$.
\item[(3)] 
$H_1$ is of type C if and only if $H_2$ is of type C.
\end{enumerate}
\end{corollary}

\begin{proof} We take account of the behavior of the nontrivial 
skew primitives described in Proposition \ref{xxpro3.4}. In all 
three types, nontrivial skew primitives exist. In types A and B, 
each skew primitive quasi-commutes with its weight, while in 
type C, no nontrivial skew primitive quasi-commutes with its 
weight. Statement (3) follows.

In type A, at most two grouplikes are weights of nontrivial skew 
primitives, while in type B, at least three grouplikes are 
weights of skew primitives. Statements (1) and (2) now follow.
\end{proof}

%%%%%%%%%%%%%%%%%%%%%%%
\subsection{Locally affine Hopf algebras}
\label{xxsec3.3}

We recall the definition of the local affine property.

\begin{definition}
\label{xxdef3.7} Let $H$ be a Hopf algebra.
\begin{enumerate}
\item[(1)]
An element $f\in H$ is called {\it locally affine} if 
it is contained in a Hopf subalgebra that
is affine. 
\item[(2)]
Let $V$ be a subset of $H$. We say that $V$ is 
{\it locally affine} if every element in $V$ is locally affine.
\end{enumerate}
\end{definition}

\begin{lemma}
\label{xxlem3.8} 
Let $H$ be a Hopf algebra. 
\begin{enumerate}
\item[(1)]
Every finite set of locally affine elements of $H$ is 
contained in an affine Hopf subalgebra of $H$.
\item[(2)] 
The locally affine elements in $H$ form a Hopf subalgebra 
of $H$, and this Hopf subalgebra is a directed union of 
affine Hopf subalgebras.
\item[(3)]
If a subset $V\subseteq H$ is locally affine, and if $H$ 
is generated by $V$ as an algebra, then $H$ is locally affine.
\item[(4)]
\cite[Corollary 3.4]{Zh}
If $H$ is pointed, then it is locally affine.
\item[(5)]
Let $V$ be a subset of $H$ such that $\sum_{i\geq 0} kS^i(v)$
is finite dimensional for each $v\in V$. If $H$ is the $k$-span 
of $V$, then $H$ is locally affine.
\item[(6)]
If $S$ has finite order, then $H$ is locally affine.
As a consequence, if $H$ is commutative or cocommutative,
then it is locally affine. 
\end{enumerate}
\end{lemma}

\begin{proof} (1) This follows from the observation that 
if $\Omega_1,\dots,\Omega_n$ are affine Hopf subalgebras 
of $H$, then the subalgebra of $H$ generated by 
$\bigcup_{i=1}^n \Omega_i$ is an affine Hopf subalgebra.

(2) If $f,g\in H$ are locally affine, then in view of part (1), 
$f\pm g$ and $fg$ are locally affine. Of course, the identity 
$1\in H$ is locally affine, because $1\in k$. Therefore the 
set $L$ of locally affine elements of $H$ is a subalgebra of 
$H$. Any finite subset of $L$ is contained in an affine Hopf 
subalgebra $\Omega$ of $H$ by (1), and $\Omega \subseteq L$ 
by definition of $L$. Hence, $L$ is a directed union of affine 
Hopf subalgebras of $H$. In particular, $L$ is a Hopf subalgebra.

(3) This is clear from part (2).

(4) This was proved by Zhuang \cite{Zh}, and we give a different
proof below.

We show, by induction, that $f\in C_n(H)$ is locally affine
where $\{C_n(H)\}_{n\geq 0}$ is the coradical filtration of $H$. 
Suppose $n=0$. Since $H$ is pointed, $C_0(H)=kG$ for a group
$G$. It is clear that $kG$ is locally affine. Now suppose that 
$C_{n-1}(H)$ is locally affine and let $f\in C_n(H)$ for some 
$n\geq 1$. By \cite[Theorem 5.4.1]{Mo}, $f=\sum_{g,h \in G(H)} f_{g,h}$
where $\Delta(f_{g,h})=f_{g,h}\otimes g+h\otimes f_{g,h}+w_{g,h}$
for some $w_{g,h}\in C_{n-1}\otimes C_{n-1}$. It suffices to show 
that each $f_{g,h}$ is locally affine, so assume that $f = f_{g,h}$. 
It is clear that $f$ is locally affine if and only if $x f$ is 
locally affine for some (or any) grouplike element $x$. By replacing 
$f$ by $xf$ for some grouplike $x$, we can assume that
$\Delta(f)=f\otimes 1+ g\otimes f+w$ where $g$ is 
grouplike and $w\in C_{n-1}\otimes C_{n-1}$. By the
antipode axiom, $\epsilon(f)=S(f)+g^{-1} f+ w_0$
where $w_0\in C_{n-1}^2$, or $S(f)=-g^{-1} f+v$
for some $v\in C_{n-1}^2$.  By part (2) and the induction
hypothesis, all tensor components of $w$ are contained 
in an affine Hopf subalgebra of $H$, say $\Omega$. In particular,
$v\in \Omega$. Let $B$ be the subalgebra of $H$ generated by 
$g^{\pm 1}$, $f$ and $\Omega$. Then
$B$ is an affine Hopf subalgebra of $H$. Since $f\in B$, 
$f$ is locally affine. The assertion follows by induction.

(5) Let $f\in H$ and let $W$ be a finite dimensional subcoalgebra 
of $H$ containing $f$. By hypothesis, $X:=\sum_{i\geq 0}
S^i(W)$ is finite dimensional. Then $f$ is contained in
the affine subalgebra $k\langle X\rangle$ which is a Hopf
subalgebra as $X$ is a subcoalgebra with $S(X)\subseteq X$.

(6) This is a consequence of part (5). (Recall from 
\cite[Corollary 1.5.12]{Mo} that $S^2 = \Id$ if $H$ is 
commutative or cocommutative.)
\end{proof}

\begin{proposition}
\label{xxpro3.9}
Let $H$ be a Hopf algebra and $K$ a locally affine Hopf 
subalgebra of $H$. If $H$ is generated by $K$ and 
$x\in H$ as an algebra, then $H$ is locally affine.
\end{proposition}

\begin{proof} Let $f\in H$. Since $H$ is generated by
$K$ and $x$, there is a finite dimensional subspace $V
\subseteq K$ such that 
\begin{enumerate}
\item[(a)]
$f, S(x)\in k\langle x, V\rangle=:A$, and 
\item[(b)]
$\Delta(x)\in A\otimes A\subseteq H\otimes H$.
\end{enumerate}
Since $V\subseteq K$ is finite dimensional, there is an affine 
Hopf subalgebra $K_0\subseteq K$ which contains
$V$. Let $H_0$ be the subalgebra of $H$ generated by 
$K_0$ and $x$. By definition, $A\subseteq H_0$ and $H_0$ is
affine. Moreover,
$f\in H_0$. By (a) and (b), $S(x)\in H_0$ and $\Delta(x)\in H_0\otimes
H_0$. Since $K_0$ is a Hopf subalgebra of $K$, it is easy to see 
that $H_0$ is a Hopf subalgebra of $H$. Therefore $H$ is locally affine.
\end{proof}

Given a Hopf algebra $H$, note that $\Ext^1_H(k,k)\cong (\fm/\fm^2)^*$ where 
$\fm=\ker \epsilon$ \cite[Lemma 3.1(a)]{GZ}. So 
$\Ext^1_H(k,k)\neq 0$ if and only if $\fm\neq \fm^2$.

\begin{proposition}
\label{xxpro3.10} 
Let $H$ be a Hopf algebra domain of GK-dimension two 
satisfying $(\natural)$. If $H$ is locally affine,
then it is pointed.
\end{proposition}

\begin{proof} Let $V$ be a simple subcoalgebra of $H$, and 
let $f\in \fm\setminus \fm^2$. By Lemma \ref{xxlem3.8}(1),
there is an affine Hopf subalgebra $K\subseteq H$ that
contains $V$ and $f$. Since $f\in \fm_K \setminus \fm_K^2$
where $\fm_K:=\ker \epsilon_K$, $K$ satisfies 
$(\natural)$. By \cite[Theorem 0.1]{GZ}, $K$
is pointed, whence $\dim_k V = 1$. Therefore $H$ is pointed.
\end{proof}

We finish this section with a well-known lemma.

\begin{lemma}
\label{xxlem3.11}
Let $H$ be a Hopf algebra with countable dimensional
$C_1(H)$. If $H$ is generated {\rm(}as an algebra\/{\rm)} by $C_1(H)$, then 
$H$ is a union of an ascending chain of affine 
Hopf subalgebras, each of which is finitely
generated by its grouplikes and skew primitives. 
\end{lemma}

\begin{proof} Let $G=\{g_i\}_{i\in I}$ be the group 
of grouplikes in $H$ and $C=\{y_j\}_{j\in J}$ a 
set of nontrivial skew primitive elements with weights
in $G$. Then $\bigcup_{i\in I} g_i C$ spans the 
space $C_1(H)$. Since $\dim_k C_1(H)$ is countable,
so are $I$ and $J$. We list elements in $G$ and $C$ 
$$G=\{g_1,\cdots, g_n, \cdots\}\qquad {\text{and}}\qquad 
C=\{y_1,\cdots, y_n,\cdots\}.$$ 
Let $B\langle n\rangle$
be the Hopf subalgebra of $H$ generated by $y_1,\cdots, y_n$
and a finite set of grouplike elements containing 
$g_1^{\pm 1}, \cdots, g_n^{\pm 1}$ and all $x^{\pm 1}$
where $x$ appears in the expression $\Delta(y_i)$ for some 
$i=1,\cdots,n$. 
Then $B\langle n\rangle$ is an affine Hopf subalgebra of $H$
and $H=\bigcup_n B\langle n\rangle$. We may choose the $B\langle n\rangle$ so that $B\langle n\rangle \subseteq B\langle n+1\rangle$ for all $n$.
\end{proof}

%%%%%%%%%%%%%%%%%%%%%%%%%%%%%%%%%%%%%%%%%
%%%%%%%%%%%%%%%%%%%%%%%
\section{Classification results}
\label{xxsec4}

In this section we prove a couple of classification theorems 
for Hopf domains of GK-dimension two.

%%%%%%%%%%%%%%%%%%%%%%%
\subsection{Classification in the pointed case with $(\natural)$}
\label{xxsec4.1}

We start with a classification of pointed Hopf domains 
of GK-dimension two satisfying $(\natural)$.

\begin{lemma}
\label{xxlem4.1}
Let $H$ be a locally affine Hopf domain of GK-dimension two 
satisfying $(\natural)$, and assume that $\GKdim G(H) = 1$. 
Then $\dim_k H = \aleph_0$.
\end{lemma}

\begin{proof} Obviously $H$ is infinite dimensional.

By Lemma \ref{xxlem3.8}(2), $H$ is a directed union of 
affine Hopf subalgebras $K_\alpha$, and we may assume 
that all of them have GK-dimension two. We may also 
assume that $G(K_\alpha)$ is nontrivial, whence 
$\GKdim G(K_\alpha) = 1$.  Because of $(\natural)$, there 
is some $f \in \mf{m} \setminus \mf{m}^2$, where 
$\mf{m} = \ker \epsilon$, and we may assume that 
$f \in K_\alpha$ for all $\alpha$. Then 
$f \in  \mf{m}_{K_\alpha} \setminus \mf{m}_{K_\alpha}^2$, 
where $\mf{m}_K = \ker \epsilon_K$, so that $K_\alpha$ 
satisfies $(\natural)$. Now by \cite[Theorem 0.1]{GZ}, 
each $K_\alpha$ is of type A, B, or C. In particular, 
$K_\alpha$ is generated by its grouplikes and skew 
primitives, so the same holds for $H$.

In view of Corollary \ref{xxcor3.5}, we see that, for 
any grouplike $g \in H$, the space of $(1,g)$-skew 
primitive elements in $H$ is at most $2$-dimensional. 
Since $H$ has only countably many grouplikes 
(Lemma \ref{xxlem1.1}), there is a countable dimensional 
subspace $V$ of $H$ that contains $G(H)$ and all skew 
primitives. Therefore $H = k\langle V\rangle$ has 
countable dimension.
\end{proof}

\begin{theorem}
\label{xxthm4.2}
Let $H$ be a pointed Hopf domain of GK-dimension two 
satisfying $(\natural)$. Suppose $H$ is not affine 
{\rm(}or equivalently, not noetherian{\rm)}. Then it is isomorphic 
to one of the following.
\begin{enumerate}
\item[(1)]
$kG$ where $G$ is a subgroup of ${\mathbb Q}^{2}$ 
containing ${\mathbb Z}^2$ that is not finitely
generated.
\item[(2)]
$kG$ where $G = L\rtimes_\varphi R$ for some subgroup
$L$ of 
${\mathbb Q}$ containing ${\mathbb Z}$
and some subgroup $R$ of $\ZZ_{(2)}$
containing ${\mathbb Z}$, and at least one of $L$ or $R$
is not finitely generated.
\item[(3)]
$A_{G}(e,\chi)$ where $G$ is a non-cyclic subgroup of
${\mathbb Q}$.
\item[(4)]
$C_{G}(e,\tau)$ where $G$ is a non-cyclic subgroup of
${\mathbb Q}$.
\item[(5)]
$B_G(\{p_i\},\chi)$ where $G$ is a non-cyclic subgroup of
${\mathbb Q}$.
\end{enumerate}
\end{theorem}

\begin{proof} 
The cases when $kG(H)$ has GK-dimension zero or two are 
done by Proposition \ref{xxpro3.1}. Now assume 
$\GKdim G(H) = 1$. We also assume that $H$ is not 
isomorphic to any Hopf algebra in parts (1)--(4), and we 
will prove that it is isomorphic to one of those in part (5). 
By Corollary \ref{xxcor1.5}, it remains to consider the case
when $H$ contains a Hopf subalgebra isomorphic to $A(1,1)$.

In view of Lemma \ref{xxlem1.1}(3), $G(H)$ is isomorphic to 
a non-cyclic subgroup $G$ of $\QQ$. Write $G(H)$ in 
the form $\{ x^a \mid a \in G \}$ as in Notation \ref{xxsec0.1}. Since $H$ contains a copy of $A(1,1)$, 
there are a grouplike $x$ and a nontrivial skew primitive $y$ 
in $H$ such that
$$xy=yx, \quad {\text{and}}\quad \Delta(y)=y\otimes 1+ x\otimes y.$$
After replacing $G$ by an isomorphic subgroup of $\QQ$ if 
necessary, we may assume that $1 \in G$ and $x = x^1$.

Since $H$ is pointed, it is locally affine by Lemma \ref{xxlem3.8}(4). 
Thus, by Lemma \ref{xxlem4.1} and its proof, $H$ is the union 
of an increasing sequence of affine Hopf subalgebras
$$ K_1 \subsetneqq K_2 \subsetneqq \cdots \,, $$
each being one of type A, B, or C from Notation \ref{xxnot3.3}. 
Since the Hopf subalgebra $k\langle x^{\pm1},y \rangle \cong A(1,1)$ 
is contained in some $K_j$, Corollary \ref{xxcor3.6} implies that 
none of the $K_i$ is of type C. From the same corollary, we find 
that either all the $K_i$ are of type A or all but finitely many 
$K_i$ are of type B. Since we may delete any $K_i$ that does not 
properly contain  $k\langle x^{\pm1},y \rangle$, there is no loss 
of generality in assuming that
$$
k\langle x^{\pm1},y \rangle \subsetneqq K_1 \,.
$$
From Proposition \ref{xxpro3.4} and the details of 
\cite[Constructions 1.1, 1.2, 1.4]{GZ}, we see that in each 
$K_i$, there is at most one grouplike $g_i$ which is the weight 
of a nontrivial skew primitive that commutes with $g_i$. 
Consequently, taking also Corollary \ref{xxcor3.5} into account,
\begin{enumerate}
\item[(i)] 
The unique grouplike $g \in G(H)$ which is the weight of a 
nontrivial skew primitive that commutes with $g$ is $g=x$. 
The space of $(1,x)$-skew primitives in $H$ is $ky + k(1-x)$.
\end{enumerate}

Suppose first that all the $K_i$ are of type A. If $K_i \cong A(n_i,q_i)$ 
with $q_i=1$ or $q_i$ not a root of unity, it follows from 
Proposition \ref{xxpro3.4}(1) that $y$ is, up to a scalar, the 
only nontrivial skew primitive element modulo $kG(K_i)$ in $K_i$.
As a consequence, $y$ is the only nontrivial skew primitive 
element modulo $kG(H)$ in $H$. By Lemma \ref{xxlem3.2}, $H$ is 
either type A or type C. This yields a contradiction. 

Thus, after deleting some of the $K_i$, we may assume that 
$K_1 \cong A(n_1,q_1)$ where $q_1$ is a primitive $d_1$-th 
root of unity for some $d_1>1$. By Proposition \ref{xxpro3.4}(2), 
$K_1 = k\langle x_1^{\pm1}, z_1\rangle$ for some grouplike 
$x_1$ and some nontrivial skew primitive $z_1$ with weight 
$x_1^{n_1}$ such that $x_1^{d_1n_1} = x$ and 
$z_1^{d_1} \in ky + k(1-x)$. Since $K_1$ then contains two 
nontrivial skew primitives with different weights, so do 
all the $K_i$, and another application of the proposition 
yields $K_i \cong A(n_i,q_i)$ where $q_i$ is a primitive 
$d_i$-th root of unity for some $d_i>1$. Moreover, 
$K_i= k \langle x_i^{\pm1}, z_i\rangle$ for some grouplike 
$x_i$ and some nontrivial skew primitive $z_i$ with 
weight $x_i^{n_i}$ such that $x_i^{d_i n_i} = x$ and 
$z_i^{d_i} \in ky + k(1-x)$. Further, 
$x_{i+1}^{n_{i+1}} = x_i^{n_i}$ and 
$z_{i+1} \in kz_i + k(1-x_i^{n_i})$, from which we see 
that $K_{i+1} = k\langle x_{i+1}^{\pm1}, z_i\rangle$.

At this point, $H$ is generated by $G(H) \cup \{z_1\}$. 
Let $\chi$ be the character of $G(H)$ determined by
$$g^{-1} z_1 g=\chi(g) z_1+\tau(g) (1-x_1^{n_1}),\qquad
\forall \; g\in G:=G(H).$$
Since $\chi$ is non-trivial, one can choose $\tau(g)=0$
for all $g$ by \cite[Lemma 2.2(c)]{WZZ2}. 
Then there is a surjective Hopf algebra map $\phi:
A_G(x_1^{n_1},\chi)\to H$. But $A_G(x_1^{n_1}, \chi)$
is a domain of GK-dimension two, so $\phi$ is an 
isomorphism, contradicting one of our assumptions.

Therefore, all but finitely many $K_i$ are of type B. 
After deleting the exceptions, we may assume that all 
$K_i$ are of type B.

Each $K_i$ is now generated by grouplikes $x_i^{\pm1}$ and 
finitely many nontrivial skew primitives, say $y_{ij}$ for 
$j \in J_i$. From the details of \cite[Construction 1.2]{GZ}, 
we have positive integers $n_i$, $p_{0i}$, and $p_{ij}$ for 
$j \in J_i$ and some $q_i \in \kx$ such that for all $j,l \in J_i$,
\begin{enumerate}
\item[(ii)] 
$|J_i| \ge2$ and $p_{ij} \ge 2$.
\item[(iii)] 
$p_{0i} \mid n_i$ and $p_{0i}$ together with the $p_{ij}$ 
are pairwise relatively prime.
\item[(iv)] 
$q_i$ is a primitive $l_i$-th root of unity, where 
$l_i = m_in_i/p_{0i}$ and $m_i = \prod_{j\in J_i} p_{ij}$, 
and $q_i^{m_{ij}^2 n_i}$ is a primitive $p_{ij}$-th root of unity.
\item[(v)] 
$x_i y_{ij} x_i^{-1} = q_i^{m_{ij}} y_{ij}$, where $m_{ij} = m_i/p_{ij}$.
\item[(vi)] 
$y_{ij} y_{il} = y_{il} y_{ij}$ and $y_{ij}^{p_{ij}} = y_{il}^{p_{il}}$.
\item[(vii)] 
$y_{ij}$ has weight $x_i^{m_{ij} n_i}$, and these elements do not commute.
\item[(viii)] 
$y_{ij}^{p_{ij}}$ is a nontrivial skew primitive element with 
weight $x_i^{m_i n_i}$, and these elements commute.
\end{enumerate}
In view of (i), it follows that $x_i^{m_i n_i} = x$ and 
$y_{ij}^{p_{ij}} \in ky + k(1-x)$. In particular, 
$$
1/m_in_i \in G \qquad\text{and}\qquad x_i = x^{1/m_in_i}.
$$
Set $G_i := \ZZ (1/m_in_i)$, so that $G(K_i) = \{ x^c \mid c \in G_i \}$.

Define a character $\chi_i$ on the group $G_i(1/m_i) = \ZZ (1/m_i^2n_i)$ 
so that $\chi_i(1/m_i^2n_i) = q_i$, and observe that
$$
x^{c/m_in_i} y_{ij} x^{-c/m_in_i} = q_i^{cm_{ij}} y_{ij} 
= \chi_i(c/m_in_ip_{ij}) y_{ij} \qquad \forall\; c\in\ZZ, \; j \in J_i \,.
$$

Temporarily set $t := i+1$, and consider $j \in J_i$. Since $y_{ij}$ 
is a nontrivial skew primitive element of $K_t$ with weight 
$x_i^{m_{ij} n_i}$ and $y_{ij}$ does not commute with its weight, 
Proposition \ref{xxpro3.4}(3) implies that there is some 
$s\in J_t$ such that $x_i^{m_{ij} n_i} = x_t^{m_{ts} n_t}$ 
and $y_{ij} \in ky_{ts} + k(1-x_t^{m_{ts} n_t})$. Now 
$x_i^{m_{ij} n_i}$ quasi-commutes with $y_{ij}$ and $y_{ts}$ 
but does not commute with these elements, while it does commute 
with $1-x_t^{m_{ts} n_t}$. Since $y_{ts} \notin kG(K_t)$, it 
follows that $y_{ij} \in ky_{ts}$. After rearranging indices, 
we may thus assume that $J_i \subseteq J_t$ and
$$
x_i^{m_{ij} n_i} = x_t^{m_{tj} n_t} \qquad\text{and}
\qquad y_{ij} = \alpha_j y_{tj} \;\; \text{with} 
\;\; \alpha_j \in \kx, \qquad \forall\; i \in J_i \,.
$$
Note that $x^{1/p_{ij}} = x_i^{m_{ij} n_i} 
= x_t^{m_{tj} n_t} = x^{1/p_{tj}}$ implies $p_{ij} = p_{tj}$. 
For $j,l \in J_i$, we have 
$$
(\alpha_j y_{tj})^{p_{tj}} = y_{ij}^{p_{ij}} 
= y_{il}^{p_{il}} = (\alpha_l y_{tl})^{p_{tl}},
$$
and for $r\in J_t \setminus J_i$ we may choose $\alpha_r \in \kx$ 
such that $\alpha_r^{p_{tr}} = \alpha_j^{p_{tj}}$, so that 
$(\alpha_r y_{tr})^{p_{tr}} = (\alpha_j y_{tj})^{p_{tj}}$. 
Hence, we may replace all the generators $y_{tu}$ of $K_t$ 
by the elements $\alpha_u y_{tu}$. This means there is no 
loss of generality in assuming that
$$
x_i^{m_{ij} n_i} = x_{i+1}^{m_{i+1,j} n_{i+1}}, \qquad 
p_{ij} = p_{i+1,j} \,, \qquad\text{and}\qquad 
y_{ij} = y_{i+1,j} \qquad \forall\; j \in J_i \,.
$$

If $d_i$ denotes the product of the $p_{i+1,r}$ for 
$r\in J_t \setminus J_i$ (where an empty product equals $1$), 
then $m_{i+1} = d_im_i$ and $m_{i+1,j} = d_i m_{ij}$ for 
$j \in J_i$. Since $x_i \in G(K_{i+1}) = \langle x_{i+1}\rangle$, 
we have $x_i = x_{i+1}^{e_i}$ for some nonzero integer 
$e_i$, whence $m_{i+1} n_{i+1} = e_i m_i n_i$. For any $j \in J_i$,
$$
q_i^{m_{ij}} y_{ij} = x_i y_{ij} x_i^{-1} 
= x_{i+1}^{e_i} y_{i+1,j} x_{i+1}^{-e_i} 
= q_{i+1}^{e_i m_{i+1,j}} y_{i+1,j} = q_{i+1}^{d_i e_i m_{ij}} y_{ij} ,
$$
whence $q_i^{m_{ij}} = q_{i+1}^{d_i e_i m_{ij}}$. Since 
the GCD of $\{ m_{ij} \mid j \in J_i \}$ is $1$, it follows 
that $q_i = q_{i+1}^{d_ie_i}$. Consequently,
$$
\chi_{i+1}(1/m_i^2n_i) = \chi_{i+1}( d_ie_i/ m_{i+1}^2 n_{i+1}) 
= q_{i+1}^{d_ie_i} = q_i \,,
$$
and therefore $\chi_{i+1}$ restricted to $G_i(1/m_i)$ equals $\chi_i$.

Now $G = \bigcup_{i=1}^\infty G_i = \bigcup_{i=1}^\infty \ZZ(1/m_in_i)$, and 
$$
M := \sum_{i=1}^\infty \sum_{j\in J_i} \Znonneg (1/p_{ij}) 
= \bigcup_{i=1}^\infty \Znonneg (1/m_i),
$$
whence
$$GM =  \sum_{i=1}^\infty \sum_{j\in J_i} G(1/p_{ij}) 
= \bigcup_{i=1}^\infty \Znonneg (1/m_i^2 n_i) = \bigcup_{i=1}^\infty G_i(1/m_i).
$$
Consequently, there is a well defined character $\chi$ on 
$GM$ which restricts to $\chi_i$ on $G_i(1/m_i)$ for all $i$. The set
$$
P := \{ p_{ij} \mid i \in \Znonneg, \; j \in J_i \}
$$
is a set of pairwise relatively prime integers $\ge2$, and 
each $1/p_{ij} = m_{ij} n_i / m_in_i \in G$. Moreover, the scalar
$$
\chi(1/p_{ij}^2) = \chi_i(m_{ij}^2 n_i/m_i^2 n_i) = q_i^{m_{ij}^2 n_i}
$$
is a primitive $p_{ij}$-th root of unity.

Our data now satisfy all the conditions required to define 
the Hopf algebra $B_G(P,\chi)$ as in \S\ref{xxsec2.1}, and 
there is a surjective Hopf algebra map
$\pi: B_G(P,\chi) \rightarrow H$ sending the generators 
$x^{\pm a}$, $y_{ij}$ of $B_G(P,\chi)$ to the elements with 
the same names in $H$. Since both $B_G(P,\chi)$ and $H$ are 
domains of GK-dimension two, $\pi$ is an isomorphism. 
Therefore $H$ is isomorphic to a Hopf algebra in part (5).
\end{proof}

%%%%%%%%%%%%%%%%%%%%%%%
\subsection{Removing the ``pointed'' hypothesis}
\label{xxsec4.2}
Our next goal is to prove Theorem \ref{xxthm4.2} 
without assuming that $H$ is pointed. By Proposition
\ref{xxpro3.10}, it suffices to show $H$ is locally affine.

\begin{lemma}
\label{xxlem4.3}
Let $H$ be a Hopf domain of GK-dimension two satisfying 
\nat. If $H$ is not commutative, then there 
is a quotient Hopf algebra $K:=H/I$ that is a commutative 
domain of GK-dimension one. Furthermore, $K$ is one of 
the Hopf algebras listed in Lemma {\rm\ref{xxlem1.1}}.
\end{lemma}

\begin{proof} Let $\fm=\ker \epsilon$ and let $I
=\bigcap_{i\geq 1} \fm^i$. By \cite[Lemma 4.7]{LW}, 
$I$ is a Hopf ideal. 

Let $e(H)$ be the dimension of $\Ext^1_H(k,k)$ and 
let $\gr H= \bigoplus_{i=0}^{\infty} \fm^i/\fm^{i+1}$. 
By \cite[Proposition 3.4(b)]{GZ}, $e(H)\leq \GKdim H=2$. 
Hence $e(H)=1$ or $2$. By \cite[Proposition 3.4(a)]{GZ}, 
$\gr H\cong U({\mathfrak g})$ where ${\mathfrak g}$ is a
graded Lie algebra generated in degree 1 and 
$\dim_k \gfrak_1 = e(H)$. If $e(H)=1$,
then $\gr H=k[x]$ which is commutative. If $e(H)=2$, 
then ${\mathfrak g}_1={\mathfrak g}$ by 
\cite[Proposition 3.4(b)]{GZ}. So ${\mathfrak g}$ is 
abelian and again $\gr H\cong U({\mathfrak g})$ is commutative. 

By \cite[Lemma 3.5]{GZ}, $H/I$ is commutative. Since $H$ 
is a domain of GK-dimension two and $H$ is not commutative,
$I\neq 0$ and $H/I$ has GK-dimension at most one. On the 
other hand, $\Ext^1_H(k,k) \ne 0$ implies $\mfrak > \mfrak^2$, 
and consequently $\Ext^1_{H/I}(k,k) \ne 0$. Another 
application of \cite[Proposition 3.4(b)]{GZ} yields 
$\GKdim H/I \ge e(H/I) > 0$, and thus $\GKdim H/I = 1$.
Moreover,
$\gr (H/I)$ is an enveloping algebra and thus a domain. 
Therefore $H/I$ is a 
commutative domain of GK-dimension one. The assertion 
follows. 
\end{proof}

\begin{theorem}
\label{xxthm4.4}
Let $H$ be a Hopf domain of GK-dimension two 
satisfying $(\natural)$. Then $H$ is locally affine.
\end{theorem}

\begin{proof}
If $H$ is commutative, the assertion follows by Lemma 
\ref{xxlem3.8}(6). From now on, assume that $H$ is not commutative.
By Lemma \ref{xxlem4.3}, there is a Hopf ideal $I$ such that 
$K:=H/I$ is a Hopf domain of GK-dimension one. By Lemma 
{\rm\ref{xxlem1.1}}, we are in one of the following two cases:

{\bf Case 1:} $K= kG$ where $G$ is a nonzero subgroup of $\QQ$.

{\bf Case 2:} $K = k[t]$ is a polynomial ring, with 
$\Delta(t) = t\otimes 1 + 1\otimes t$.

The following analysis is similar to the one in \cite{GZ}. In 
fact, the ideas and arguments are copied from \cite{GZ}. 
Let $\pi: H \rightarrow K$ be the quotient map, set
$$\rho := (\id\otimes\, \pi)\Delta : H\to H\otimes K 
\qquad\text{and}\qquad 
\lambda :=(\pi\otimes \id)\Delta : H\to K\otimes H,$$
and note that $H$ becomes a right (resp., left) comodule 
algebra over $K$ via $\rho$ (resp., $\lambda$), see 
\cite[Subsection 4.1]{GZ}.

{\bf Case 1.} Write $K = kG$ in the form 
$\bigoplus_{a\in G} k x^a$. For $a\in G$, let 
$$H_a := \{h\in H\mid \rho(h)=h\otimes x^a\} 
\qquad\text{and}\qquad 
{}_{a} H :=\{h\in H \mid \lambda(h)=x^a\otimes h\}.$$
Then $H$ is a $G$-graded algebra in two ways:
$$H=\bigoplus_{a\in G} H_a =\bigoplus_{a\in G} {_a H},$$
where the first decomposition is called the $\rho$-grading 
and the second is called the $\lambda$-grading. 
Let $\pi^r_{a}$ and $\pi^{l}_{a}$ be the respective
projections from $H$ onto $H_a$ and ${_a H}$ in the 
above decompositions. Then by $G$-graded versions of 
\cite[(E5.0.1) and (E5.0.2)]{GZ}, we have 
\begin{equation}
\label{E4.4.1}\tag{4.4.1}
\pi^r_a \pi^l_b=\pi^l_b \pi^r_a, \quad \forall \; a,b\in G
\end{equation}
and, writing ${_a H _b}={_a H}\cap {H_b}$ for all $a,b\in G$,
we have 
\begin{equation}
\label{E4.4.2}\tag{4.4.2}
H_b=\bigoplus_{a\in G} {_a H_b},\quad {\text{and}} \quad 
{_a H}=\bigoplus_{b\in G} {_a H_b}.
\end{equation}
In particular, these give $G$-gradings for the algebras 
$H_0$ and $_0H$, and a $(G\times G)$-grading 
$\bigoplus_{a,b \in G} {}_aH_b$ for $H$.

By the proof of \cite[Lemmas 5.2 and 5.3]{GZ}, 
$H_a \cap {}_aH \ne 0$ for each $a\in G$, and $H$ is 
strongly $G$-graded with respect to both the $\rho$-grading
and the $\lambda$-grading. Then, $G$-graded versions of 
\cite[Lemma 5.4(a)(b)]{GZ} imply that 
$\dim_k H_0 = \infty$ and $\GKdim H \ge \GKdim H_0 +1$. 
Since $H_0$ is a domain, it cannot be algebraic over $k$, 
and therefore $\GKdim H_0 = 1$.

{\bf Case 1a.} Suppose that $H_0={_0 H}$. By \cite[Lemma 4.3(c)]{GZ},
$H_0$ is a Hopf subalgebra of $H$. By Lemma \ref{xxlem1.1},
$H_0$ is either $k[t]$ or $kG'$ where $G'$ is a 
torsionfree abelian group of rank one. Hence, $H_0$ is a 
Bezout domain (see the proof of \cite[Lemma 6.2]{GZ}). Thus 
each $H_a$ is a free $H_0$-module of rank one, say 
$H_a = h_a H_0$. Since $H$ is strongly $G$-graded, $h_a$ 
must be invertible. In particular, $\eps(h_a) \ne 0$, so we 
may replace $h_a$ by $\eps(h_a)^{-1} h_a$ and thus assume 
that $\eps(h_a) = 1$.

We claim that each $h_a$ is grouplike. By a $G$-graded 
version of \cite[Lemma 5.1(b)]{GZ}, 
$\Delta(h_a)\in H_a\otimes H_a$, and so $\Delta(h_a) =
(h_a\otimes h_a)w$ for some $w\in H_0\otimes H_0$. Since
$h_a$ is invertible, so is $w$, and 
$(\eps\otimes\id)(w) = (\id\otimes\, \eps)(w) = 1$ by 
the counit axiom applied to $h_a$. By
\cite[Lemma 4.4(a)]{GZ}, $w$ is a homogeneous invertible
element of $H_0\otimes H_0$. If $H_0=k[t]$, then 
$w=c1\otimes 1$ for some $c\in \kx$,
while if $H_0= kG'$, then $w= c g \otimes g'$ for some 
$c\in \kx$ and $g,g' \in G'$. In either case, it follows 
from the equations 
$(\eps\otimes\id)(w) = (\id\otimes\, \eps)(w) = 1$ that 
$w= 1\otimes 1$. Therefore $h_a$ is grouplike, as claimed.

Now $S^2(h_a) = h_a$ for all $a\in G$. Since $H_0$
is commutative, $S^2$ is the identity on $H_0$. As
$H$ is generated by $H_0$ and the $h_a$, we find that 
$S^2$  is the identity on $H$. Therefore by 
Lemma \ref{xxlem3.8}(6), $H$ is locally affine.

{\bf Case 1b.} Suppose that $H_0\neq {_0 H}$. Either 
$H_0 \nsubseteq {}_0H$ or ${}_0H \nsubseteq H_0$, say 
$H_0 \nsubseteq {}_0H$. Then $_aH_0 \ne 0$ for at least 
one nonzero $a\in G$. By a $G$-graded version of 
\cite[Lemma 5.4(c)]{GZ}, $\dim_k {}_aH_0 \le 1$ for 
all $a \in G$. For any $a,b \in G$, multiplication by 
a nonzero element of $_{-b}H_{-b}$ embeds $_aH_b$ in 
$_{a-b}H_0$, whence $\dim_k {_a H_b}\leq 1$. 
By a $G$-graded version of \cite[Lemma 5.1(d)]{GZ}, 
$S^2({_a H_b})\subseteq {_a H_b}$ for all $a,b\in G$. 
This implies that $\sum_{i\geq 0}
S^i(_aH_b)$ has dimension at most two. Since $H$ is 
spanned by the $_aH_b$, we conclude by Lemma 
\ref{xxlem3.8}(5) that $H$ is locally affine. 

{\bf Case 2.} By \cite[Lemma 8.1]{GZ}, there are two 
commuting locally nilpotent derivations $\delta_r$ 
and $\delta_l$ on $H$ such that
$$\rho(h) = \sum_{n=0}^\infty \frac1{n!} \delta_r^n(h) 
\otimes t^n \qquad\text{and}\qquad 
\lambda(h) = \sum_{n=0}^\infty \frac1{n!} t^n 
\otimes \delta_l^n(h), \qquad \forall\; h\in H.$$
In particular, $H_0=\ker \delta_r$ and ${_0 H}=\ker \delta_l$. 
Following the proofs of \cite[Lemmas 9.2, 9.3]{GZ}, 
we find that $H_0$ and $_0H$ have GK-dimension one.

We claim that $H_0={_0 H}$. Suppose not, say, 
$H_0\not\subseteq {_0 H}=\ker \delta_l$. Since $\delta_l$ 
commutes with $\delta_r$, $H_0$ is $\delta_l$-invariant. 
Thus, $\delta_l$ restricts to a nonzero locally nilpotent 
derivation on $H_0$, denoted by $\delta$. Let 
$H_{00}=\ker \delta$, and choose $u \in H_0 \setminus H_{00}$. 
Then $H_{00} [u]$ is a polynomial subalgebra of $H_0$ 
(by the argument of \cite[Lemma 9.2]{GZ}), which implies that 
$\GKdim H_{00}=0$. Since $k$ is algebraically closed, 
$H_{00}=k$. Since $\ker \delta=k$, there is an element 
$u\in H_0\setminus H_{00}$ such that $\delta(u)=1$.
This implies
$\lambda (u)=1\otimes u+t\otimes 1$, whereas $\rho(u)=
u\otimes 1$ as $u\in H_0$. Set $y=\pi(u)\in K$ and 
compute $\Delta(y)$ in
the following two ways:
$$\begin{aligned}
\Delta(y)&=(\pi\otimes \pi) \Delta(u)=(\pi\otimes \id) \rho(u)=y\otimes 1,\\
\Delta(y)&=(\pi\otimes \pi) \Delta(u)=(\id\otimes\, \pi) \lambda(u)
=1\otimes y+t\otimes 1.
\end{aligned}
$$
The counit axioms then yield $y = \epsilon(y)$ and $y= \epsilon(y)+t$, giving a contradiction. Therefore
we have proved that  $H_0= {_0 H}$.

Since $H_0={_0 H}$,  by \cite[Lemma 4.3(c)]{GZ}, $H_0$ is a Hopf 
subalgebra of $H$. By \cite[Theorem 8.3(b)]{GZ}, $H$ has 
the form $H=H_0[x; \partial]$, 
which is generated by $H_0$ and $x$. Since $H_0$ is 
commutative, $H_0$ is locally affine by 
Lemma \ref{xxlem3.8}(6). So $H$ is locally affine by 
Proposition \ref{xxpro3.9}. 

Combining Cases  1 and 2, we have that $H$ is locally affine.
\end{proof}

Now we are ready to prove Theorem \ref{xxthm0.1}.

\begin{proof}[Proof of Theorem {\rm\ref{xxthm0.1}}]
By Theorem \ref{xxthm4.4}, $H$ is locally affine.
By Proposition \ref{xxpro3.10}, $H$ is pointed.
Therefore the assertion follows from Theorem 
\ref{xxthm4.2}.
\end{proof}

We also have a slight modification of Theorem 
\ref{xxthm0.1}. 

\begin{proposition}
\label{xxpro4.5} 
Let $H$ be a Hopf domain of GK-dimension two.
Suppose $H$ is a union of an ascending chain
of Hopf subalgebras $\{K_i\}_{i=1}^{\infty}$ 
such that all $K_i$ satisfy
$(\natural)$. Then $H$ satisfies $(\natural)$.
As a consequence, $H$ is isomorphic to 
one of Hopf algebras in Theorem 
{\rm{\ref{xxthm0.1}}}.
\end{proposition}

\begin{proof} By Theorem \ref{xxthm0.1} each $K_i$
is pointed and generated by grouplikes and 
skew primitives. Hence $H$ is pointed and 
generated by grouplikes and 
skew primitives.

It is well-known that all
algebras in Proposition \ref{xxpro3.1} 
satisfy $(\natural)$. Therefore we
can assume that $\GKdim C_0(H)
=\GKdim G(H)=1$. 

Now repeating the proof of
Theorem \ref{xxthm4.2} gives the result. 
The hard part is concerning type B, where
we give an alternative proof as below.

If all $K_i$ are of type B, let $I_i$ be the Hopf
ideal of $K_i$ generated by all nontrivial skew 
primitives that quasi-commute with their weight.
Then $K_i/I_i\cong kG(K_i)$, which induces an injection 
$K_i/I_i\subseteq K_{i+1}/I_{i+1}$ for all $i$.
It is easy to see that $I_i=I_{i+1}\cap K_i\subseteq I_{i+1}$. 
Let $I=\bigcup_i I_i$. Then $I$ is a Hopf ideal of $H$ 
and $H/I$ is a union of Hopf subalgebras isomorphic to $K_i/I_i$. 
As a consequence, $H/I\cong kG(H)$ where $G(H)$ is a 
nontrivial subgroup of $\QQ$. Since $kG(H)$ satisfies
$(\natural)$ (Lemma \ref{xxlem1.1}), so does $H$. 
\end{proof}

%%%%%%%%%%%%%%%%%%%%%%%%%%%%%%%%%%%%%%%%%
%%%%%%%%%%%%%%%%%%%%%%%
\section{Other properties}
\label{xxsec5}
In this section we will prove Corollary \ref{xxcor0.2} 
and Theorem \ref{xxthm0.5}. We use some ideas of 
Takeuchi \cite{Ta1}. Some parts of the proofs were also 
suggested by Quanshui Wu \cite{Wu}. We would like to thank Wu for 
sharing his comments and proofs with us.

%%%%%%%%%%%%%%%%%%%%%%%
\subsection{Takeuchi's idea}
\label{xxsec5.1}
In this subsection we review some ideas of Takeuchi \cite{Ta1}.
The following lemma was proved in \cite{Ta1} 
in the commutative case and the proof works for a general Hopf algebra.
Let $H^{+}$ be the kernel of the counit of $H$. 

\begin{lemma}
\label{xxlem5.1} \cite[Lemma 3.9]{Ta1}
Let $H$ be a Hopf algebra. Suppose that $K$ and  $K'$ are Hopf subalgebras
of $H$ such that $K' \subseteq K$. Then there is a right $H$-module 
isomorphism:
$$K \otimes_{K'} H \xrightarrow{\tau} (K/{K(K^{\prime +})}) \otimes H, 
\quad \tau : x \otimes y \mapsto \sum {\bar x_1}\otimes x_2y,$$
with the inverse map
$$(K/{K(K^{\prime +})}) \otimes H \xrightarrow{\mu}
K \otimes_{K'} H, \quad \mu : {\bar u} \otimes 
v \mapsto \sum u_1 \otimes S(u_2) v.$$
\end{lemma}

\begin{proof} Since the proof is very nice, we include it here.

For any $x \in K$, $z \in K'$, and $y \in H$, we have
$$
\begin{aligned}
\sum x_1z_1 \otimes x_2z_2y - \sum x_1 \otimes x_2zy
=&\sum x_1z_1 \otimes x_2z_2y - \sum x_1\epsilon (z_1) \otimes x_2z_2y \\
=&\sum x_1(z_1 - \epsilon (z_1)) \otimes x_2z_2y \in KK^{\prime +} \otimes H.
\end{aligned}
$$
This implies that the map $\tau$ is well-defined by the definition of 
the tensor product.
On the other hand, for any $u=xz \in K(K^{\prime +})$, where 
$x \in K$, $z \in K^{\prime +}$, and $v \in H$, we have
$$\begin{aligned}
\sum u_1 \otimes_{K'} S(u_2)v &= \sum x_1z_1 \otimes_{K'} S(x_2z_2)v
=  \sum x_1z_1 S(z_2) \otimes_{K'} S(x_2)v \\
&= \sum x_1 \epsilon (z) \otimes_{K'} S(x_2)v =0.
\end{aligned}
$$
Hence $\mu$ is well defined. 

It is easy to see that
$$\begin{aligned}
\mu \tau (x \otimes_{K'} y)&=\mu \biggl(\sum {\bar x_1}\otimes x_2y \biggr) 
=\sum x_1 \otimes_{K'} S(x_2)x_3y\\
&= \sum x_1 \otimes_{K'} \epsilon (x_2)y=x \otimes_{K'} y,
\end{aligned}
$$
and
$$\tau \mu ({\bar u} \otimes v) =\tau \biggl(\sum u_1 \otimes_{K'} S(u_2)y \biggr)
=\sum {\bar u_1} \otimes u_2 S(u_3)v ={\bar u} \otimes v.$$
Therefore $\tau$ is invertible with the inverse $\mu$.  
\end{proof}

There is also a left $H$-module version of the above lemma.
The next lemma is also well-known (even before \cite{Ta1})
and is used in the study of ``faithfully flat descent''
when $A$ is commutative. Since the proof is short, it is 
included here.

\begin{lemma}
\label{xxlem5.2}
Suppose that $B \subseteq A$ is a ring extension such that $A_B$ 
(or ${}_BA$) is faithfully flat. Then the sequence $0 \longrightarrow B 
\longrightarrow A \xrightarrow{ \; f \; } A \otimes_B A$, 
where $f$ is the map sending $x \mapsto x \otimes 1 - 1 \otimes x$, 
is exact.
\end{lemma}

\begin{proof}
Let $C=\ker(f)=\{x\in A\mid x\otimes 1=1\otimes x\}$. It is clear that 
$B \subseteq C$.  Since $A_B$ is flat, there are embeddings
$$A=A\otimes_B B\hookrightarrow A\otimes_B C \hookrightarrow A\otimes_B A.$$
For any element $a\otimes c\in A\otimes_B C$, we obtain that
$$ a \otimes c=a(1\otimes c)=a(c\otimes 1)=ac\otimes 1,$$
which implies that the map 
$(A=)A\otimes_B B\to A\otimes_B C$
is surjective. Consequently, $A\otimes_B C/B=0$. Since $A_B$ is faithfully 
flat, $C/B=0$ and $C=B$ as desired. 
\end{proof}

\begin{proposition}
\label{xxpro5.3} Let $H$ be a Hopf algebra and $K\subsetneq H$
be a Hopf subalgebra of $H$. Suppose that 
$H_{K}$ is faithfully flat. Then $HK^{+}\neq H^{+}$.
\end{proposition}

\begin{proof} We modify Takeuchi's proof \cite{Ta1}.
By Lemma \ref{xxlem5.2}, the sequence
$$0 \to K \longrightarrow H \xrightarrow{ \; f \; } H \otimes_{K} H$$
is exact, where $f$ is the map $x \mapsto x \otimes 1 - 1\otimes x$. 
By Lemma \ref{xxlem5.1}, 
$$\tau: x \otimes y \mapsto \sum {\bar x_1} \otimes x_2y$$ 
gives rise to an isomorphism $H \otimes_{K} H \to H/{HK^{+}} 
\otimes H$. Hence
$$0 \to K\longrightarrow H \xrightarrow{\tau\circ f} H/{HK^{+}} 
\otimes H,$$ where 
$\tau \circ f: x \mapsto \sum {\bar x_1} \otimes x_2 - 1\otimes x$,
is exact. Since $K\neq H$, the map $\tau\circ f\neq 0$, 
which implies that $H/HK^{+}\neq k$. This completes the proof. 
\end{proof}

%%%%%%%%%%%%%%%%%%%%%%%
\subsection{Some consequences}
\label{xxsec5.2}
If $H$ is pointed, then it satisfies (FF) \cite{Ra2}.
We now prove Theorem \ref{xxthm0.5}.

\begin{theorem}
\label{xxthm5.4} Let $H$ be a left noetherian Hopf algebra.
\begin{enumerate}
\item[(1)]
If $H$ satisfies {\rm{(FF)}}, then $H$ is of $S$-finite type. 
As a consequence, $\dim_k H$ is countable.
\item[(2)]
If $H$ is locally affine and satisfies {\rm{(FF)}}, 
then $H$ is affine. 
\item[(3)]
If $H$ is pointed, then $H$ is affine.
\end{enumerate}
\end{theorem}

\begin{proof} (1) Suppose $H$ is not of $S$-finite type.
Let $K_0=k$ and define a sequence of $S$-finite type Hopf 
subalgebras $K_n$ inductively. Suppose $K_n$ is generated
by $\bigcup_{i=0}^{\infty} S^i(V_n)$ where $V_n$ is a finite
dimensional subcoalgebra of $H$. Let $W$ be a 
finite dimension subcoalgebra of $H$ such that $W\nsubseteq K_n$.
Let $V_{n+1}=V_{n}+W$ and $K_{n+1}$ be the Hopf subalgebra
generated by $\bigcup_{i=0}^{\infty} S^i(V_{n+1})$. By definition,
$K_n\neq K_{n+1}$ and each $K_{n}$ is of $S$-finite type.
Let $K=\bigcup_{n} K_n$. Then $K$ is a Hopf subalgebra that
is not of $S$-finite type. Since $H$ satisfies (FF), $H_K$ 
is faithfully flat. Then $K$ is left noetherian. We may replace
$H$ by $K$ and assume that $H=\bigcup_{n} K_n$ without loss of 
generality. 

Since $H$ is left noetherian, there is an $N$ such that 
$H K_n^{+}=HK_N^{+}$ for $n\geq N$. Since $H=\bigcup_{n} K_n$,
$$HK_N^{+}=\bigcup_{n\geq N} HK_n^{+}=H \bigcup_{n\geq N} K_n^{+}
=HH^{+}=H^{+}.$$
By Proposition \ref{xxpro5.3}, $H^{+}
\neq HK_N^{+}$, yielding a contradiction. Therefore, $H$ is of
$S$-finite type. 

Any Hopf algebra of $S$-finite type is countably generated 
and so has countable $k$-dimension.

(2) This follows from part (1) and the fact $S$-finite type plus 
local affineness imply that $H$ is affine. 

(3) If $H$ is pointed, then it is (FF) by \cite{Ra2}. By 
Lemma \ref{xxlem3.8}(4), $H$ is locally affine. The assertion
follows from part (2).
\end{proof}

%%%%%%%%%%%%%%%%%%%%%%%%%%%%%%%%%%%%%%%%%
%%%%%%%%%%%%%%%%%%%%%%%
\subsection{Global dimension}
\label{xxsec5.3}

In this subsection we will prove Corollary \ref{xxcor0.2}.
The following lemma is known.

\begin{lemma}\cite[Corollary 1]{Be}, \cite[Proposition 2.1]{Os}
\label{xxlem5.5} Let $A$ be an algebra and $\{A(n)\}_{n=1}^\infty$ 
an ascending chain of subalgebras of $A$ such that
$A=\bigcup_n A(n)$. Then
%\begin{enumerate}
%\item[(1)]
%$$\injdim A\leq \max\{\injdim A(n)\mid \forall \; n\}+1.$$
%\item[(2)] 
$$\gldim A\leq \max\{\gldim A(n)\mid \forall \; n\}+1.$$
%\end{enumerate}
\end{lemma}

\begin{proof}[Proof of Corollary {\rm\ref{xxcor0.2}}]
(1) and (2) follow by construction.

(3) This is a consequence of \cite[Theorem A(ii)]{Sk}.

(4) This follows by \cite{Ra2}.

(5) If $H$ is in Theorem \ref{xxthm0.1}(1-5),
then every affine Hopf subalgebra $K$ of $H$ of GK-dimension 2
has global dimension 2 by the proof of 
\cite[Proposition 0.2(1)]{GZ}. Since $H_K$ and $_KH$ are free (see part (4)), 
by \cite[Theorem 7.2.6]{MR},
$\gldim H\geq \gldim K=2$. By Lemmas \ref{xxlem3.11} and 
\ref{xxlem5.5}, $\gldim H\leq 3$.

If $H$ is in Theorem \ref{xxthm0.1}(6),
then there is an affine Hopf subalgebra $K$ of $H$ of GK-dimension 2
which has global dimension $\infty$ by the proof of 
\cite[Proposition 0.2(1)]{GZ}. Since $H_K$ and $_KH$ are free (see part (4)), 
by \cite[Theorem 7.2.8(i)]{MR},
$\gldim H\geq \gldim K=\infty$. 
\end{proof}

%%%%%%%%%%%%%%%%%%%%%%%

%%%%%%%%%%%%%%%%%%%%


\begin{thebibliography}{10}

\bibitem[Be]{Be}
I. Berstein, 
\emph{On the dimension of modules and algebras. IX. Direct limits}, 
Nagoya Math. J.  {\bf 13} (1958) 83--84. 


%\bibitem[Br2]{Br2} 
%K.A. Brown, 
%\emph{Representation theory of Noetherian Hopf algebras satisfying a 
%polynomial identity}, in Trends in the Representation Theory of 
%Finite Dimensional Algebras (Seattle 1997), (E.L. Green and B.
%Huisgen-Zimmermann, eds.), Contemp.  Math. \textbf{229}
%(1998), 49--79.

%\bibitem[Br3]{Br3}
%\bysame,
%\emph{Noetherian Hopf algebras},
%Turkish J. Math. \textbf{31} (2007), suppl., 7--23.

\bibitem[BG]{BG}
K.A. Brown and P. Gilmartin, 
\emph{Hopf algebras under finiteness conditions}, 
Palest. J. Math.  {\bf 3} (2014),  Special issue, 356--365.

%\bibitem[BGo]{BGo}
%K.A. Brown and K.R. Goodearl, 
%\emph{Homological aspects of Noetherian PI Hopf algebras 
%of irreducible modules and maximal dimension}, 
%J. Algebra \textbf{198} (1997), no. 1, 240--265.


\bibitem[BO]{BO} K.A. Brown, S. O'Hagan, J.J. Zhang and G. Zhuang, 
\emph{Connected Hopf algebras and iterated Ore extensions},
J. Pure Appl. Algebra, \textbf{219} (2015), 2405--2433.

\bibitem[Ch]{Ch}
A. Chirvasitu, 
\emph{Cosemisimple Hopf algebras are faithfully flat over Hopf 
subalgebras}, Algebra Number Theory {\bf 8} (2014),  no. 5, 
1179--1199. 

%\bibitem[Go1]{Go1}
%K.R. Goodearl,
%\emph{Prime ideals in skew polynomial rings and quantized Weyl algebras}, J. Algebra
%\textbf{150} (1992), 324--377.

\bibitem[Go]{Go}
K.R. Goodearl, 
\emph{Noetherian Hopf algebras}, Glasg. Math. J. 
{\bf 55}  (2013),  no. A, 75--87. 

%\bibitem[GL]{GL}
%K.R. Goodearl and E.S. Letzter, 
%\emph{Prime ideals in skew and $q$-skew polynomial rings}, 
%Memoirs Amer. Math. Soc.  No.~521 (1994), vi+1--106.

\bibitem[GZ]{GZ}
K.R. Goodearl and J.J. Zhang, \emph{Noetherian Hopf algebra domains of Gelfand-Kirillov 
dimension two}, J. Algebra \textbf{324} (2010), 3131--3168.

\bibitem[Gr]{Gr} 
M. Gromov, 
\emph{Groups of polynomial growth and expanding maps}, 
Inst. Hautes {\'E}tudes Sci. Publ. Math. No. {\bf 53} 
(1981), 53--73.


\bibitem[LW]{LW}
D.-M. Lu, Q.-S. Wu and J.J. Zhang, 
\emph{Homological integral of Hopf algebras}, 
Trans. Amer. Math. Soc. {\bf 359} (2007) 4945--4975.

%\bibitem[Mol]{Mol}
%R. Molnar, 
%\emph{A commutative Noetherian Hopf algebra over a field 
%is finitely generated}, 
%Proc. Amer. Math. Soc. {\bf 51}  (1975), 501--502.

\bibitem[MR]{MR}
J.C. McConnell and J.C. Robson, 
\emph{Noncommutative Noetherian Rings}.
 Wiley, Chichester, 1987.

\bibitem[Mo]{Mo} 
S. Montgomery, 
\emph{Hopf Algebras and their Actions on Rings}, 
CBMS Regional Conference Series in Mathematics 82, Providence, R1, 1993.

\bibitem[NZ]{NZ}
W.D. Nichols and M.B. Zoeller,
\emph{A Hopf algebra freeness theorem},
Amer. J. Math. {\bf 111}, (1989) 381--385.

\bibitem[Os]{Os}
B.L. Osofsky, 
\emph{Upper bounds on homological dimensions}, 
Nagoya Math. J. {\bf 32} (1968), 315--322. 


\bibitem[Pa]{Pa} 
A.N. Panov, 
\emph{Ore extensions of Hopf algebras}, Mat. Zametki, 
{\bf 74} (2003), 425--434.

\bibitem[Ra1]{Ra1}
D.E. Radford, 
\emph{Operators on Hopf algebras}, Amer. J. Math. 
{\bf 99} (1977), 139--158.

\bibitem[Ra2]{Ra2}
D.E. Radford, 
\emph{Pointed Hopf algebras are free over Hopf subalgebras}, 
J. Algebra {\bf 45} (1977), no. 2, 266--273. 

\bibitem[Sc]{Sc}
P. Schauenburg, 
\emph{Faithful flatness over Hopf subalgebras: Counterexamples}, 
in ``Interactions Between Ring Theory and Representations of Algebras'' 
(Murcia 1998), (F. Van Oystaeyen and M. Saor\'\i n, Eds.), New York (2000) Dekker, 
pp. 331--344.

\bibitem[Sk]{Sk}
S. Skryabin, 
\emph{New results on the bijectivity of antipode of a Hopf algebra},
J. Algebra  {\bf 306}  (2006),  no. 2, 622--633. 

\bibitem[Ta1]{Ta1}
M. Takeuchi,
\emph{A correspondence between Hopf ideals and sub-Hopf algebras}, 
Manuscripta Math. {\bf 7}  (1972), 251--270. 

\bibitem[Ta2]{Ta2}
M. Takeuchi,
\emph{Relative Hopf modules---equivalences and freeness criteria},
J. Algebra {\bf 60} (1979) 452--471.

\bibitem[WZZ1]{WZZ1}
D.-G. Wang, J.J. Zhang and G Zhuang, 
\emph{Hopf algebras of GK-dimension two with vanishing 
Ext-group},
J. Algebra {\bf 388} (2013) 219--247.

\bibitem[WZZ2]{WZZ2}
D.-G. Wang, J.J. Zhang and G Zhuang, 
\emph{Lower bounds of growth of Hopf algebras},
Trans. Amer. Math. Soc. {\bf 365} (2013), no. 9, 4963--4986.


\bibitem[WZZ3]{WZZ3}
D.-G. Wang, J.J. Zhang and G Zhuang, 
Primitive cohomology of Hopf algebras, preprint, (2014)
arXiv:1411.4672.

%\bibitem[Wa4]{Wa4}
%D.-G. Wang, J.J. Zhang and G Zhuang, 
%\emph{Classification of connected Hopf algebras of GK-dimension four},
%Trans. Amer. Math. Soc. \textbf{367} (2015), 5597--5632.

\bibitem[Wu]{Wu}
Q.-S. Wu, \emph{Private communication}, (2015).

%\bibitem[WZ1]{WZ1}
%Q.-S. Wu and J.J. Zhang, 
%\emph{Gorenstein property of Hopf graded algebras}, 
%Glasg. Math. J. {\bf 45}  (2003),  no. 1, 111--115. 

\bibitem[WZ]{WZ}
Q.-S. Wu and J.J. Zhang, 
\emph{Noetherian PI Hopf algebras are Gorenstein}, 
Trans. Amer. Math. Soc. {\bf 355} (2002), 1043--1066.


\bibitem[Zh]{Zh} 
G. Zhuang, 
\emph{Properties of pointed and connected Hopf algebras of finite 
Gelfand-Kirillov dimension}, 
J. London Math. Soc. (2)  {\bf 87}  (2013),  no. 3, 877--898.



\end{thebibliography}
\end{document}